\documentclass[a4paper, reqno]{amsart}
\usepackage[utf8]{inputenc}
\usepackage{palatino,newpxmath} 
\usepackage[left=3cm,right=3cm,top=3cm,bottom=3cm]{geometry}
\usepackage[english]{babel}
\usepackage{natbib}
\usepackage{url}
\usepackage{amsmath}

\usepackage{amssymb}
\usepackage{amsthm}
\usepackage{mathtools}
\usepackage{amsfonts}
\usepackage{bbm}
\usepackage{caption}
\usepackage{gensymb}
\usepackage{tikz-cd}
\usepackage{mathdots}
\usepackage{cancel}
\usepackage{multicol}
\usepackage[font=small,labelfont=bf]{caption}
\usepackage[colorlinks=true, urlcolor = blue]{hyperref} 
\usepackage[shortlabels]{enumitem}
\usepackage{paracol}
\setenumerate[0]{label=(\arabic*)}

\usepackage[many]{tcolorbox}  

\newtcolorbox{boxE}{
    enhanced, 
    boxrule = 0pt, 
    borderline = {0.75pt}{0pt}{main}, 
    borderline = {0.75pt}{2pt}{sub} 
}

\newcommand{\Z}{\mathbb{Z}}

\newcommand{\R}{\mathbb{R}}
\newcommand{\C}{\mathbb{C}}
\newcommand{\D}{\mathbb{D}}

\renewcommand{\P}{\mathbb{P}}

\newcommand{\F}{\mathbb{F}}

\renewcommand{\phi}{\varphi}

\DeclareMathAlphabet{\mathpzc}{OT1}{pzc}{m}{it} 

\newcommand{\q}{\mathfrak{q}}

\newcommand{\gl}{\mathfrak{gl}}

\usepackage{mathrsfs}

\renewcommand{\O}{\mathscr{O}}
\newcommand{\M}{\mathcal{M}}

\newcommand{\RN}[1]{%
  \textup{\uppercase\expandafter{\romannumeral#1}}%
}

\newcommand{\J}{\mathbb{J}}

\newcommand{\Conn}{\mathfrak{Con}^V_\Theta}

\theoremstyle{definition}
\newtheorem*{defn}{Definition}
\newtheorem{es}{Example}

\newtheorem{oss}{Remark}

\newtheorem*{fact}{Fact}

\theoremstyle{plain}
\newtheorem{thm}{Theorem}[section]
\newtheorem{cor}[thm]{Corollary}
\newtheorem{prop}[thm]{Proposition}
\newtheorem{lemma}[thm]{Lemma}

\newtheorem{theoremA}{Theorem}

\begin{document}


\title{}
\title[Moduli Space of Connections]{Moduli Space of Connections on\\Rational Irregular Curves}
\author{Mattia Morbello}

\begin{abstract}
   We study the compactification of the moduli space of a certain class of rank-two irregular connections on the Riemann sphere, presenting one double pole and two simple poles. To construct the compactification explicitly, we identify a class of such irregular connections with the data of a rational irregular curve together with an extra complex parameter. \\As a first step, we compactify the moduli space of rational irregular curves using a technique inspired by the Kapranov's compactification of the spaces $\M_{0,n}$. We then introduce the notion of irregular stable nodal curve to describe the curves lying on the boundary components, in the spirit of the work of Deligne and Mumford.\\ Second, we study the behaviour of the extra complex parameter to complete the compactification, obtaining a three dimensional quasi-projective variety $\Conn$. This variety comes with a fibration $\pi\colon \Conn\to\P^1$ whose fibers $\pi^{-1}(t)$, for $t\in \C^*$, are the Okamoto spaces of initial conditions for the equation Painlevé V. 
\end{abstract}

\maketitle

\tableofcontents


\thispagestyle{empty}
\restoregeometry

\newpage

\section*{Introduction}

In 1979 K. Okamoto \cite{okamoto} studied the different spaces of initial conditions for all the Painlevé equations, corresponding to some 8-blow-ups of the second Hirzebruch surface $\F_2$. Two years later M. Jimbo, T. Miwa and K. Ueno \cite{JIMBO1981306} established a link between the Okamoto's work and the theory of isomonodromic deformations. More recently, in 2006, P. Boalch \cite{BOALCH2001137} studied, in a more analytic framework, the moduli spaces of irregular meromorphic connection and their isomonodromic deformations, with an emphasis on symplectic structures. The same year, M. Inaba, K. Iwasaki, and M-H. Saito \cite{inab} used Okamoto's results to study the geometry of the Painlevé VI equation in a completely algebraic setting: they constructed the fine moduli space of stable parabolic connection on $\P^1$ with logarithmic poles and deeply studied the Riemann-Hilbert map and the isomonodromic foliation. 

\begin{center}
    \includegraphics[width=6cm]{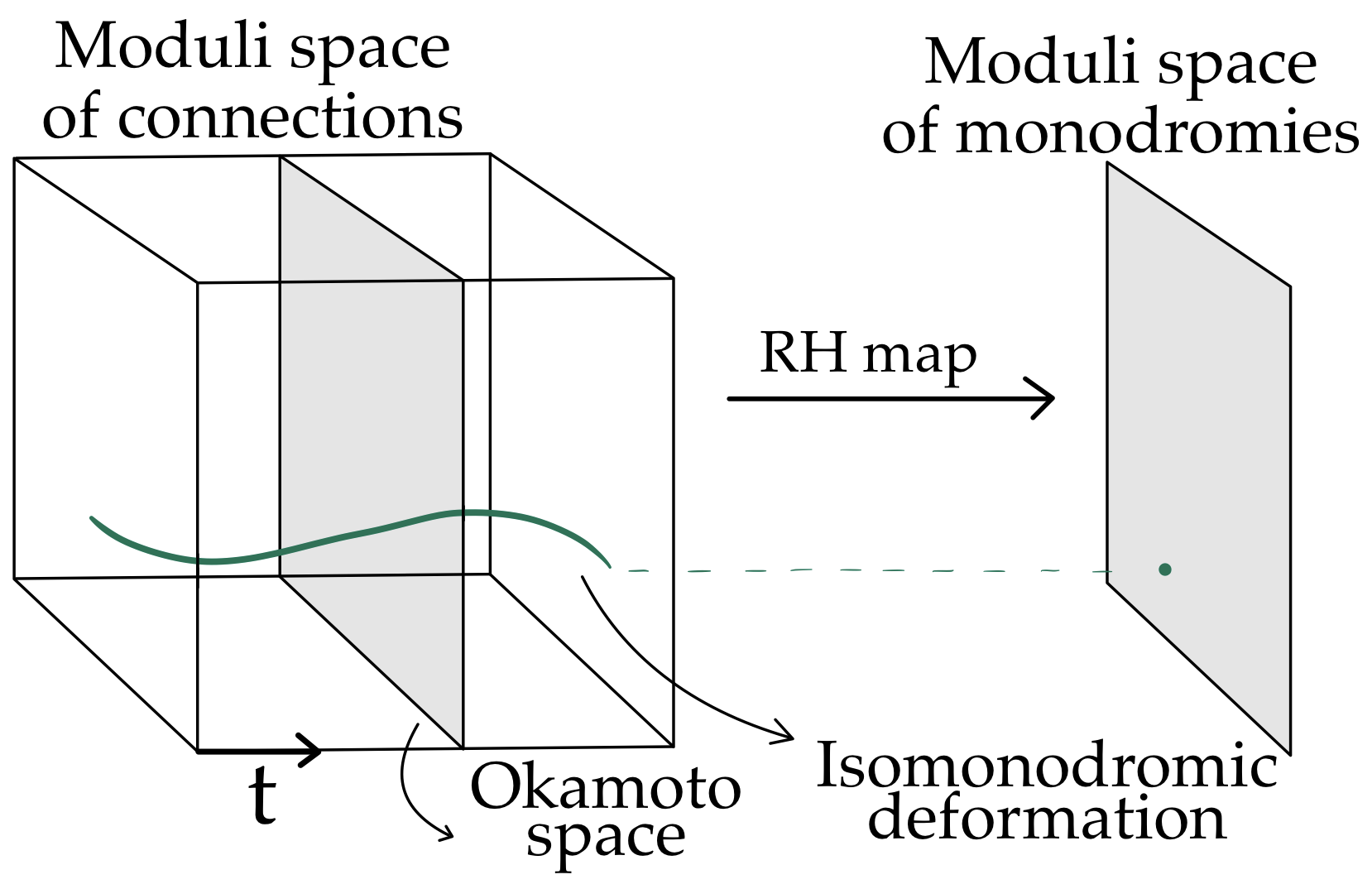}
\end{center}
Our work fits within the geometrical framework of the irregular Riemann-Hilbert correspondence, following the recent developments in \cite{BOALCH2001137}, \cite{saitovanderput}, \cite{Chiba}, and \cite{paulramis}, with the difference that in the last two papers weighted projective spaces are used to construct a compact moduli space. In particular, we study the moduli space of a certain class of rank-two meromorphic connections on the Riemann sphere with one double pole and two simple poles, which we refer as PV connections. In the literature, they are sometimes called confluent Heun equations. 

\medskip

We denote by $\Conn$ the three-dimensional moduli space of PV connections. Our main goal is to construct a compactification $\overline\Conn$ that extends the Okamoto one. In particular we show that this space comes with a fibration $\pi\colon\overline\Conn\to \P^1$ whose fibers $\pi^{-1}(t)$, for $t\in \C^*$, are the Okamoto spaces of initial conditions for the equation Painlevé V. We describe in Theorem \ref{thm:fibtA} what the fibers $\pi^{-1}(0)$ and $\pi^{-1}(\infty)$ look like, extending the previous works on the moduli spaces of such connections. In a future work currently being drafted we use this compactification to study the isomonodromic foliation on the boundary components.

\medskip

To construct $\Conn$ and its compactification, we use the normal form presented in \cite{Diarra}. It is proved that, up to meromorphic gauge transformations, rank-two meromorphic connections on the Riemann sphere can be written in companion form, and explicit formula depending on the residual spectral data at $0$, $1$, and $\infty$ is provided. Certain free parameters, denoted $(t,q,\hat p)$, naturally appear: they will play a key role in the construction of $\Conn$.\\
Indeed, during the normalization process, an apparent singularity appears at a point $x=q$. Its residual matrix has integerl eigenvalues, and the eigenvector corresponding to the non-zero eigenvalue depends on the complex parameter $\hat p$. Finally, the last crucial parameter, traditionally denoted by $t$, as it represents the time variable in the Painlevé V equation, arises in the matrix associated with the irregular singularity.

One can then prove that for any choice of $(t,q,\hat p)\in (\C\setminus\{0\})\times(\P^1\setminus\{0,1,\infty\})\times\C$ there exists a unique isomorphism class of PV connections fitting these complex parameters. However, this construction does not capture all such connections, since the approach mentioned before requires $q\neq 0,1,\infty$, whereas this condition is not always satisfied. Nevertheless, this suggests that the space $U:=(\C\setminus\{0\})\times(\P^1\setminus\{0,1,\infty\})\times\C$ provides a natural starting point for constructing the moduli space we seek and its compactification. Note that connections lying in $U$ are hence precisely the connections admitting the normal form.

To construct and compactify the moduli space, we propose a geometric interpretation of the parameters $(t,q,\hat p)$: we show that, under the action of a Möbius transformation on the base $\P^1$, the complex parameter $\hat p$ remains unchanged; the parameter $q$, as expected, transforms as a point of $\P^1$; while $t$ transforms as an element of the tangent space of $\P^1$ at $x=1$.

We call such a sphere with four marked point and a marked tangent vector a rational irregular curve; they may be regarded as a generalization of the five-punctured sphere. In this perspective, a PV connection can be viewed as the data of a rational irregular curve (identified, up to Möbius transformations, by the parameters $q$ and $t$) together with an additional complex parameter $\hat p$.

To compactify the moduli space of connections we first compactify $\M$, the moduli space of rational irregular curves, adapting a technique that Kapranov \cite{kapranov} used for the compactification of the spaces $\M_{0,n}$. The main result, proved in Section \ref{subs:compM}, is the following:
\begin{theoremA}\label{A}
   	The compactification $\overline\M$ of the moduli space of rational irregular curves is isomorphic to the singular surface obtained by the contraction of the unique $[-2]$-curve in the weak del Pezzo surface of degree five $\widetilde\M$. The singularity obtained is of $A_1$-type.
\end{theoremA}

\begin{center}
    \includegraphics[width=8cm]{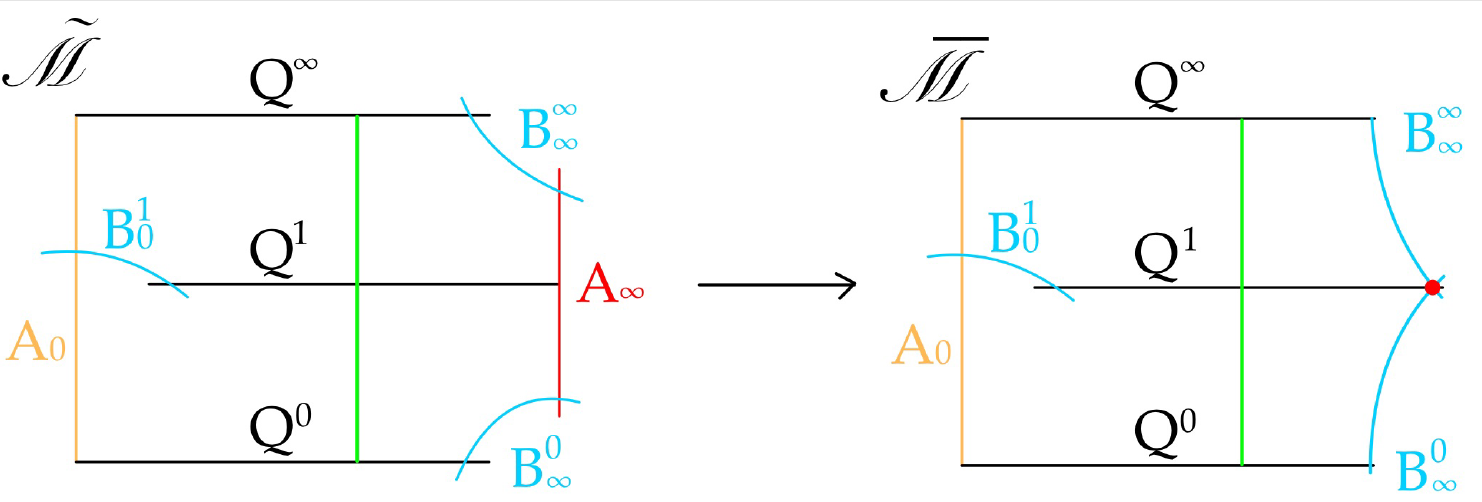}
\end{center}
A central part of this work is to understand the irregular curves appearing on the boundary components. In particular, inspired by the Deligne-Mumford work \cite{Del}, we introduce the notion of stable nodal irregular curve and we present the universal curve associated to the fine moduli space $\overline\M$.

\medskip

The next step is then to extend the line bundle structure of $U=\M\times\C$ to $\overline\M$. To avoid working on a singular surface, we first study the extension on $\widetilde \M$.

\begin{theoremA}\label{thm:B}
	The trivial line bundle $U=\M\times \C$ extends to $\widetilde \M$ as $\O_{\widetilde{\mathcal M}}(D)$, where $D\in \mathrm{Div}(\widetilde \M)$ is the divisor
	\[D=\mathrm{div}(\hat p)\equiv A_0+2Q^1+2B_0^1-B_\infty^0-B_\infty^\infty,\]
	where $A_0$, $Q^1$ and the $B^i_j$'s are some boundary components. Moreover, the line bundle $\O_{\widetilde{\mathcal M}}(D)$ is trivial when restricted to the boundary components $Q^i$s, and to the $[-2]$-curve $A_\infty$.
\end{theoremA}

To prove this result, we analyse the global trivialising section $\hat p =1$ of $\M\times\C$ and we consider the behaviour of its extension along the boundary components, studying its zeros and poles. 

\medskip

The triviality of the line bundle over the $Q^i$s allows us to define some special global holomorphic sections whose equations are determined by the residual spectral data of the connection. 
Some computations show that the total space of $\O_{\widetilde{\mathcal M}}(D)$ does not represent the moduli space we are looking for, and we prove the following.

\begin{theoremA}
	Let $X:=\P(\O\oplus\O_{\widetilde{\mathcal M}}(D))$ be the compactification of the line bundle found in Theorem \ref{thm:B}, and $\mathcal W:=\P(\O)\subseteq X$ the infinity section. The moduli space $\overline\Conn$ is constructed as follows: denote by $\widetilde X$ the blow-up of $X$ along the special sections in the $\mathcal Q^i$s, and by $\widetilde{\mathcal Q^i}$s and $\widetilde{\mathcal W}$ the strict transforms. Then
	\[\overline \Conn=\widetilde X\setminus \Big(\bigcup_{i\in\{0,1,\infty\}}\widetilde{\mathcal Q^i}\cup\widetilde{\mathcal W}\Big). \]
\end{theoremA}
\begin{center}
    \includegraphics[width=5cm]{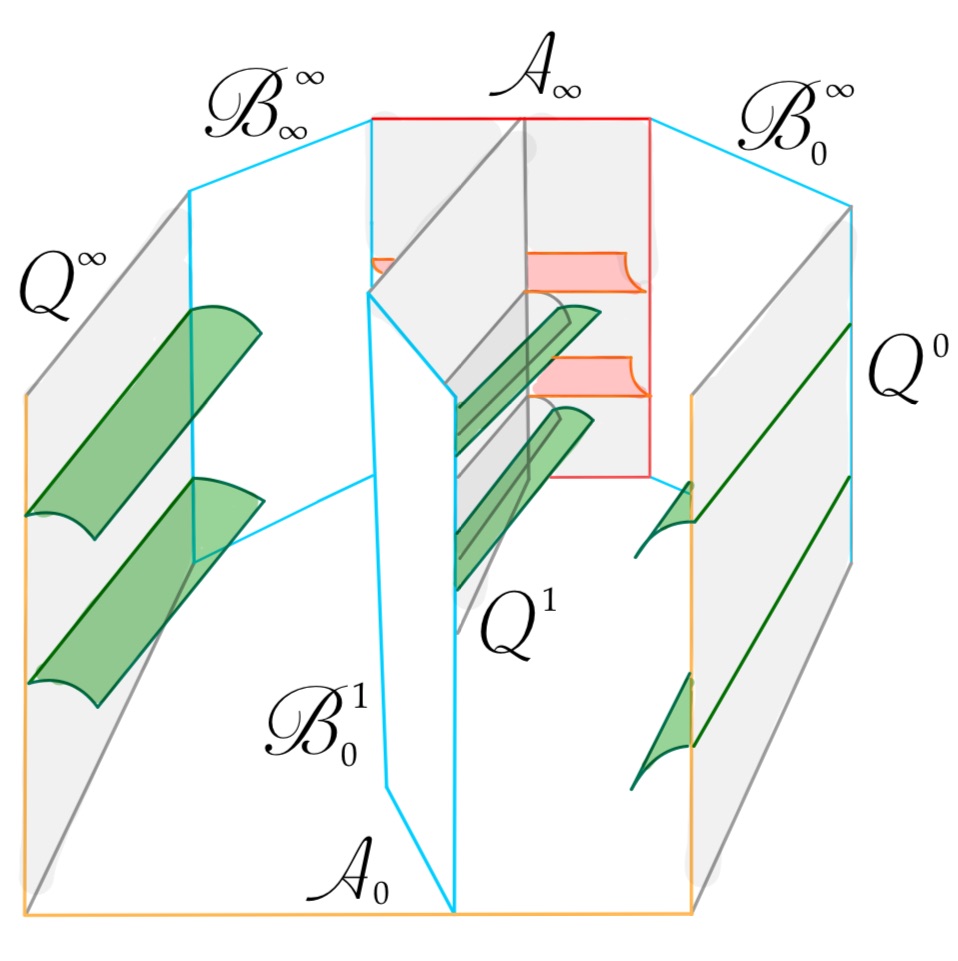}
\end{center}
We can finally describe the geometry of $\overline\Conn$ and its relation with the Okamoto spaces.
\begin{theoremA}\label{thm:fibtA}
	The compactified moduli space $\overline\Conn$ comes with a fibration $\pi\colon\overline\Conn\xrightarrow{t} \P^1$ such that
	\begin{itemize}
		\item[i)] For any $t\in \C^*$, the fiber $\pi^{-1}(t)$ is an Okamoto space, that is a 8-blow-up of the second Hirzebruch space $\F_2$.
		\item[ii)] The surface $\pi^{-1}(0)$ is given by $\mathcal A_0\cup\mathcal B^1_0$. The component $\mathcal A_0$ is isomorphic to $\F_1$, while $\mathcal B^1_0$ is isomorphic to a double blow up of $\F_1$.
		\item[iii)] Let us denote by $\mathcal F_\infty^+$ and $\mathcal F_\infty^-$ the blow up of the two special sections in $\mathcal A_\infty$. The surface $\pi^{-1}(\infty)$ is given by $\mathcal B^0_\infty\cup\mathcal B^\infty_\infty\cup\mathcal F_\infty^+\cup\mathcal F_\infty^-$. The surfaces $\mathcal B^0_\infty$ and $\mathcal B^\infty_\infty$ are isomorphic to a double blow up of $\F_1$.
	\end{itemize}
\end{theoremA}


\medskip

The text is structured as follows: in Section \ref{secprelnot} we set the notations and we recall the basic definition of the objects involved in the following. A particular attention will be given to the definition of the residual spectral data of a connection. In Section \ref{PVconn} we define the main protagonist of this work, PV connections. We describe the Riccati foliation induced by a connection on the projectivised bundle and, thanks to it, we will give an accurate geometric interpretation of the confluence of (apparent) singularities in this setting. Section \ref{S1} is devoted to describe $\M$ and its link with PV connections. Section \ref{compsect} is the main core of the work, in which a compactification of the moduli space is proposed. 

\vfill
\textit{Acknowledgements:}\\
This work is part of my Phd project, financed by the the France 2030 program, Centre Henri Lebesgue ANR-11-LABX-0020-01, and Rennes university.

I would like to thank my advisor Frank Loray for guiding me through this beautiful journey, Gabriel Calsamiglia for having spent a lot of time discussing about stable nodal curves, Sokratis Zikas for the rich discussions about birational geometry and Matilde Maccan for her mathematical and emotional support.

I received also a financial support form Collège doctorale de Bretagne, Rennes Metropole, EUR caps, IRIS-E, and the RFBM program for my mobility in Rio de Janeiro.

\newpage

\section{Preliminary Notions}\label{secprelnot}
\subsection{Linear Meromorphic Connections}\label{intro1}
We are interested in the study of a special class of linear meromorphic connections, of PV type, but let us start by introducing what a meromorphic connection is in a general setting. 
\begin{defn}
    Let $X$ be a complex manifold. A linear meromorphic connection is a triple $(E, D, \nabla)$, where 
    \begin{itemize}
    \item $E\to X$ is a holomorphic vector bundle, 
        \item $D$ is an effective divisor of $X$,
        \item $\nabla$ is a $\C$-linear application $\nabla\colon \mathcal E \to \mathcal E\otimes \Omega^1_X(D),$
    satisfying the Leibniz rule given by the $\O_X$-module structure of the sheaf $\mathcal E$ of holomorphic sections: $\nabla (f\cdot \sigma)=df\cdot\sigma+f\cdot\nabla\sigma$.
    \end{itemize} 
    The divisor $D$ is called \emph{polar divisor} and it describes the order and the position of the poles that the connection must have. We recall the notation $\Omega^1_X(D):=\Omega^1_X\otimes\O_X(D)$.
\end{defn}
Even if some of the notions we introduce hold for any complex manifold and any vector bundle, for the following of the article we always suppose that $X=\P^1$ and $\mathrm{rk}E=2$. This significantly simplify notations. 
\begin{defn}
    A point $a\in \P^1$ is called a \textit{singularity} (or a \textit{pole}) for the connection $(E,D,\nabla)$  if $a\in |D|$. It is said \textit{logarithmic} if it is a simple pole. Otherwise, it is said \textit{regular} if it can be reduced to a logarithmic singularity and \textit{irregular} if not.
\end{defn}
Connections on the trivial bundle have a very explicit description.
\begin{prop}
    Any meromorphic connection $(\nabla, \O\oplus\O, D)$ on the trivial bundle can be globally expressed as 
    \[\nabla = d+\Omega\;\;\;\;\;\text{for some}\;\;\;\;\;\Omega\in \gl_2\big(\Omega^1(D)\big).\]
\end{prop}
Roughly speaking, we can see any connection on the trivial bundle as a rank 2 system of ODEs over the projective line. Indeed, if we write
\[\Omega=\begin{pmatrix}\omega_{1,1}&\omega_{1,2}\\\omega_{2,1}&\omega_{2,2}\end{pmatrix}\;\;\;\;\text{then} \;\;\;\;\nabla Y=(d+\Omega)Y=\begin{pmatrix}dy_1\\dy_2\end{pmatrix}+\begin{pmatrix}\omega_{1,1}&\omega_{1,2}\\\omega_{2,1}&\omega_{2,2}\end{pmatrix}\begin{pmatrix}y_1\\y_2\end{pmatrix}=0,\]
for some $\omega_{i,j}\in \Omega^1(D)$.\\
We wonder what happens when the vector bundle $E$ is not globally trivial. We still know that each holomorphic vector bundle is locally trivial on any contractible open set of $\P^1$, and then it is sufficient to consider the open cover $\{U_0, U_\infty\}$. In this setting, a connection is just the datum of two different system of ODEs on $\C$, that express via the formula $\nabla_0=d+\Omega_0$ and $\nabla_\infty=d+\Omega_\infty$, and the two connection matrices are related in the overlap $U_0\cap U_\infty$ by the formula
\[\Omega_0=g_{0,\infty}^{-1}\cdot\Omega_\infty\cdot g_{0,\infty}+g_{0,\infty}^{-1}\cdot dg_{0,\infty} \]
where $g_{0,\infty}$ is the cocycle of the vector bundle $E$. 
\begin{oss}
    One can easily prove this relation by noticing that if $Y_0$ is a local section in $U_0$ satisfying $\nabla_0Y_0=0$, then the section $Y_\infty=g_{0,\infty}Y_0$ defined in $U_0\cap U_\infty$ must satisfy $\nabla_\infty Y_\infty=0$.
\end{oss}
We are interested to give an explicit expression to the connection matrix in a neighbourhood of a singularity: we can then shrink a trivializing open set for $E$ in order to get a simply connected open set $U$ that contains at most one singularity at, for instance, $x=a$. On this open set $U$, the connection matrix has a very simple and explicit description:  
\[\text{if }\;\; E_{|U}\cong U\times\C^2 \;\;\text{ and }\;\;|D|\cap U=\{a\},\;\;\;\;\text{ then } \;\;\nabla_{U}=d+\Omega_U,\]
for some $\Omega_U\in \gl\big(\Omega^1(n[a])\big)$ that writes:
\[\Omega_U:=\sum_{k=1}^n A_k^{(a)}\frac{dx}{(x-a)^k} +\text{holomorphic terms}\]
where $A_k^{(a)}\in \gl_2(\C)$ are constant matrices and $n$ is the order of the pole in $x=a$. 
\begin{defn}
	We call $\Omega_{U,a}^{<0}:=\sum_{k=1}^n A_k^{(a)}\frac{dx}{(x-a)^k}$ the \textit{principal part} of the connection matrix $\Omega_U$ around the singularity $a\in \P^1$. Note that it depends on the chosen trivialization.
\end{defn}

\subsection{Horizontal Sections}
\begin{defn}
    A local section $Y$ satisfying $\nabla Y=0$ is called \textit{parallel} or \textit{horizontal} section for the connection $(E,D, \nabla)$. Equivalently, we say that $Y$ is a solution of the differential equation $\nabla=0$
\end{defn}
We are interested in studying the behaviour of such sections in a neighbourhood of a singularity $x=a$. 
\begin{thm}
    Let X be a Riemann surface and $(\nabla, E,D)$ a connection on it. Let $U\subseteq X\setminus \mathrm{Supp}(D)$ a simply connected open set trivialising $E$. Then:
    \begin{itemize}
        \item There exists a fundamental matrix of solutions $Y(x)$ defined everywhere over $U$.
        \item Any other matrix of solutions $Y'(x)$ differs from $Y(x)$ by a constant matrix factor $Y'(x)=Y(x)\cdot C$.
        \item For any $x_0\in U$, the evaluation map $Y(\cdot)\mapsto Y(x_0)$ is an isomorphism between the linear space of solutions of $\nabla_{|U}Y=0$ and $\C^2$.
    \end{itemize}
\end{thm}
\begin{proof}
    Theorem 14.1 and remark 19.2 of \cite{Ily}
\end{proof}
We can specify the theorem via the following facts coming from \cite[15.10, 20.20, 21.3]{Ily}.
\begin{fact}
	If $x=a\ni U$ is a logarithmic singularity a local basis of solution can be expressed as $Y(x)=H(x)x^A$ where $A$ is a constant matrix and $H$ is a meromorphic invertible matrix function. 
\end{fact}
\begin{fact}
	If $x=a\ni U$ is an irregular singularity, then the local solutions of $\nabla =0$ can be expressed as $Y(x)=F(x)x^A$, where $A$ is a constant matrix and $F$ is a matrix function single valued and holomorphically invertible in $U\setminus\{a\}$, but eventually having an essential singularity in $x=a$. 
\end{fact}
The description of solutions around an irregular singularity is much more complicated since Stokes phenomena appear. In fact, for simplicity, assume that $U\cong \D$ is isomorphic to the unitary complex disk centered in the singularity. We cannot express a solution as a holomorphic function in $\D^*:=\D\setminus\{0\}$, but we are only allowed to define it on angular sectors that cover $\D^*$, whose amplitude depends on the order of the singularity. For example, in the case of an order two singularity, we have two sectors:
\[V^k:=\Big\{x\in\D^*\;\;\Big|\;\;\big|\arg(x)-(\theta+k\pi)\big|<\pi\Big\}\;\;\; \text{ for }\;\; k=0,1\]
for a $\theta$ that depends on the residual spectral data of $\Omega_\D$. On $V^k$ the solutions are indeed holomorphic. Moreover, some solution $Y(x)$ will grow exponentially fast on $V^0$ and decreases faster than any finite power $|x|^{N>0}$ in $V^1$. For more details about Stokes phenomena, see paragraph 21.1 of \cite{Ily} and section 1.3 of \cite{RHrank2}.

\subsection{Equivalent Connections and Residual Spectral Data}\label{resspecd}
The goal of this work is to build the moduli space of PV type connections, introduced in the next section. A moduli space is a topological space such that any point represents the equivalence class of one object under a suitable equivalence relation. Here we define the equivalence relation that we will use to identify different meromorphic connections. 

\medskip
\textbf{Holomorphic Gauge Transformations.}
\begin{defn}
    Two connections $(E, D, \nabla)$ and $(E', D, \nabla')$ are said \textit{gauge equivalent} if there exists a holomorphic isomorphism of bundles $\Phi\in \mathrm{Iso}(E, E')$ that sends $\nabla$-horizontal sections into $\nabla'$-horizontal sections. Equivalently, we have that 
    \[\nabla=\Phi^*\nabla'.\]
\end{defn}
There is a local description of the gauge action: if $\nabla_{|U}=d+\Omega$ and $\nabla'_{|U}=d+\Omega'$, for any trivialising set $U$, it holds that
\[\Omega=\Phi^{-1}\Omega'\Phi+\Phi^{-1}d\Phi.\]
\begin{es}\label{gaugeloc}
    Let us consider a connection $(E, D, \nabla)$ expressed locally in $U$ around a double pole say centred in $x=0$. Its connection matrix has expression
    \[\Omega=A^{(0)}_2\frac{dx}{x^2}+A^{(0)}_1\frac{dx}{x}+holomorphic \;terms.\]
    A holomorphic gauge transformation can be locally written as $\Phi_{|U}=M_0+M_1x+\dots$, with $M_i\in GL_2(\C)$. The result of the gauge action is then
    \[\Big(M_0^{-1}A_2^{(0)}M_0\Big)\frac{dx}{x^2}+\Big(M_0^{-1}A_1^{(0)}M_0+M_0^{-1}A_2^{(0)}M_1-M_0^{-1}M_1M_0^{-1}A_2^{(0)}M_0\Big)\frac{dx}{x}+holomorphic \;terms.\]
\end{es}
More generally, a holomorphic gauge transformation then acts by conjugation on the highest term of the connection matrix. Therefore, if $A^{(0)}_2$ is semisimple, we can always diagonalize the principal part of a connection via a finite number of holomorphic gauge transformations.

\medskip
\textbf{Meromorphic Gauge Transformations and Elementary Transformations.}
Holomorphic gauge transformations allow us to understand when two connections on the same bundle $E$ are equivalent. We would like to have a more general definition, allowing us to establish an equivalence of connections over bundles of different degree. We should hence look into some birational bundle automorphism.
\begin{defn}
    A \textit{meromorphic gauge transformation} $\Psi\colon(\nabla,E,D)\to (\nabla',E',D')$ is a birational morphism of bundles $\Psi\colon E\to E'$ such that $\Psi^*\nabla'=\nabla$.
\end{defn}
\begin{oss}
    We do not define in this paper the notion of monodromy of a meromorphic connection, but it is important to know that meromorphic gauge transformations are monodromy preserving transformations.
\end{oss}
A particular and interesting class of meromorphic gauge transformations are elementary transformations. They are meromorphic gauge transformations $\Psi$ such that $\det\Psi$ has only one zero or one pole of order one. One can show that any meromorphic gauge transformation is a composition of elementary transformations.
\begin{defn}
	Let $E$ be a rank 2 vector bundle and $l$ be a line in $E_{x_0}$. We denote via $(E,l)$ the parabolic vector bundle. The \textit{elementary transformations} $\mathrm{elm}^\pm_{l,x_0}$ are meromorphic gauge transformations 
	\begin{center}
		\begin{tikzcd}
			{\mathrm{elm}^+_{l,x_0}\colon (E,l)} \to  (E^+,l^+) &\text{ and } & {\mathrm{elm}^-_{l,x_0}\colon (E,l)} \to  (E^-,l^-).
		\end{tikzcd}
	\end{center}
	We choose a local trivialisation $E_{|U}\cong U\times\C^2$ based in $x_0=0$ and such that $l=\C\cdot\begin{pmatrix}1\\0\end{pmatrix}$, then the elementary transformations have local expression
	\begin{equation}\label{elemtransexpr}
		\mathrm{elm}^+_{l,x_0}=\begin{pmatrix}x&0\\0&1   \end{pmatrix}\colon \C^2\to \C^2\;\;\;\;\; \text{ and }\;\;\;\;\;\mathrm{elm}^-_{l,x_0}=\begin{pmatrix}1&0\\0&\frac{1}{x}  \end{pmatrix}\colon \C^2\to \C^2
	\end{equation}
\end{defn}
\begin{prop}\label{propelmdeg}
    Let us consider a point $x_0\in U\subseteq \P^1$ and the elementary transformation $\mathrm{elm}^\pm_{l,x_0}$. It holds that $\det(E^\pm)=\det(E)\otimes\O(\pm[x_0])$ and therefore $\deg E^\pm=\deg E\pm1.$
\end{prop}
\begin{proof}
    Section 6 of \cite{Machu2007}.
\end{proof}

\begin{oss}
    When we refer to the isomorphism class of a connection, we will implicitly mean under the action of meromorphic gauge transformations. 
\end{oss}

\textbf{Residual Spectral Data.}
Cauchy's residue theorem shows the role played by the residue of a meromorphic function in computing integrals along closed path. We are looking for a similar invariant for meromorphic connections and we introduce the \textit{residual spectral data}. They will play a central role in the construction of the moduli space and in the study of its geometry. 
\begin{defn}
We call $A_1^{(a)}=\mathrm{Res}_{x=a}\Omega_U$ the \textit{residual matrix} of $\Omega_U$ at $x=a$. Note that it depends on the chosen trivialisation.
\end{defn}
We are interested in studying the spectrum of the residual matrix and, for now on, we will always suppose that it is diagonalisable. By Example \ref{gaugeloc} we see that it is in general not invariant under gauge transformations for poles of order higher than one. We therefore need a more precise algorithm to well define the residual spectral datum. 
\begin{defn}
	The \textit{residual spectral datum} of a connection $(E, D, \nabla)$ at a pole $x=a$ is computed as follows. Let $\Omega$ be any connection matrix for $(E, D, \nabla)$ in an open trivialising neighbourhood of $a$, and let $\Omega^{<0}$ be its principal part. Up to a meromorphic gauge transformation, we can diagonalize it. Let $D_1^{(a)}$ the new (diagonal) residual matrix and denote by $\{\kappa_a^+, \kappa_a^-\}$ its spectrum. It does not depend on the meromorphic gauge transformation we have chosen. We refer to the \textit{residual spectral data} of a connection $(E, D, \nabla)$ as the collection 
	\[\Big(\{\kappa_a^+,\kappa_a^-\}\Big)_{a\in|D|}.\]
\end{defn}
We would like to have an explicit formula to compute the residual spectral datum of a connection in a pole of order two $x=a$. For simplicity, we will suppose that it is a pole of order two. We can express $\mathrm{tr}\Omega=\lambda+\mu$ and $\det\Omega=\lambda\mu$ in terms of some $\lambda, \mu\in \C((x-a))$. We are interested in computing the residue of $\lambda-\mu$. It holds that
\[(\lambda-\mu)^2=(\mathrm{tr}\Omega)^2-4\det\Omega=:R(\Omega).\] 
In particular, considering the series expansion $(\lambda-\mu)^{<0}=\sum_{i=1}^2\alpha_i(x-a)^{-i}$ and $R(\Omega)^{<0}=\sum_{i=1}^4c_i(x-a)^{-i}$, it holds that
\[\left((\lambda-\mu)^2\right)^{<-2}=\frac{\alpha_2^2}{(x-a)^4}+\frac{\alpha_1\alpha_2}{(x-a)^3}=R(\Omega)^{<-2}=\frac{c_4}{(x-a)^4}+\frac{c_3}{(x-a)^3}\]
for some coefficient $\alpha_1, \alpha_2, c_3, c_4\in \C$, and we deduce the residue $\alpha_1$ of $\lambda-\mu$:
\begin{equation}\label{formresspec}
	\alpha_1=-\frac{c_3}{2\sqrt{c_4}}=\kappa_a^+-\kappa_a^-.
\end{equation}
\begin{oss}\label{fuchs}
    The residual spectral data must satisfy the Fuchs' relation
    \[\sum_{a\in |D|}\kappa_a^++ \kappa_a^-=-\deg E\]
\end{oss}
\begin{defn}
    The residual spectral data of a connection are called \textit{generic} if the conditions \
    \[\sum_{a\in |D|}\kappa_a^{\epsilon_i}\notin\Z\;\;\;\;\;\text{ and }\;\;\;\;\forall a\in |D|,\;\; \kappa_a^+-\kappa_a^-\notin \Z\]
    holds for all choices of $\epsilon_i\in\{+,-\}$.
\end{defn}
We always assume that the residual spectral data are generic. This implies that the connection is irreducible, that is, there does not exist $\nabla$-invariant line bundle $L\subseteq E$.  
\begin{oss}\label{remfuchs}
	Meromorphic gauge transformations may shift some residual data by an integer. Since the degree of the bundle changes as well, the residual spectral data still satisfy the new Fuchs' relation \ref{fuchs}. This implies that the residual spectral data of an isomorphism class of a connection are well defined only up to integers. When some residual data are zero some pathological behaviours arise, it also explains why we asked the generic condition to be satisfied.
\end{oss}

\begin{prop}\label{corO1}
    Any irreducible meromorphic connection $(E, D, \nabla)$ with $\deg D=4$ is meromorphic gauge equivalent to a connection $(\nabla', E', D')$ where $E'\cong\O(-1)\oplus\O$.
\end{prop}
\begin{proof}
     It is a consequence of the Fuchs' relation and the irreducibility of the connection (implied by the generic choice of the residual spectral data) as shown in Proposition 2.7.4 of \cite{geompanv}. An explicit description of such connections is given in Section 4.1 of \cite{Ohyama_2006}.
\end{proof}

\section{Connections of PV type}\label{PVconn}
Connections of PV type are the simplest examples of irregular meromorphic connections with polar divisor of degree four. They arise as the confluent version of the very famous and studied Painlevé VI type logarithmic connections, see for instance \cite{geompanv}. 
\begin{defn}
     A meromorphic connection $(E, D, \nabla)$ over $\P^1$ is of \textit{Painlevé type} if: $\mathrm{rank} E=2$, and the minimal polar divisor within its meromorphic gauge equivalence class has degree four.
\end{defn}
\begin{oss}
    A complete classification of Painlevé type connections can be found in Section 1.4 of \cite{Ohyama_2006}.
\end{oss}
\begin{defn}
    A Painlevé type connection is of \textit{PV type} if the minimal divisor within its meromorphic gauge equivalence class is Moebius equivalent to $[0]+2[1]+[\infty]$.
\end{defn}
Painlevé type connections with fixed minimal polar divisor and residual spectral data are the simplest example of meromorphic connections with a positive dimensional moduli space (w.r.t. meromorphic gauge equivalence). Indeed, the dimension of the moduli space depends on the degree of the polar divisor and on the rank of the vector bundle. For rank 2 connections, we need a polar divisor of degree at least four to have a moduli space of positive dimension.
    \begin{oss}
        If we ask the residual spectral data to be fixed (up to integer shifts), Euler systems are unique up to gauge equivalence (and up to a Moebius transformation that sends the two poles respectively to 0 and $\infty$). When the degree of the minimal polar divisor is three, we have the so called hypergeometric systems, and also in this case their moduli space (fixed spectral data up to integer shifts) corresponds to a point (see Section 1.3.5 of \cite{geompanv}).  
    \end{oss}

\subsection{Riccati Foliation}\label{riccati}
Any rank 2 meromorphic connection $(E, D, \nabla)$ over $\P^1$ induces a singular Riccati foliation on the projective bundle $\P(E)$. 
\begin{oss}
Given the vector bundle $E$, we can define its projective bundle $\P(E)$. The fibers of $\P(E)$ are just the projectivisation of the fibers of $E$. The cocycle of $E$ descends to the quotient, since the action is linear. Finally, given a fiber $E_a$ for $a\in \P^1$, we will denote by $(a,p)=(a,[1:p])\in F_a:=\P(E_a)$ the point representing the line $\langle(1,p)\rangle\in E_a$.
\end{oss}
Since $\nabla$ is a $\C$-linear application, it descends to the projectivisation. Indeed, in any open set $U$ trivialising the connection, we have an expression in coordinates $\nabla=d+\Omega_U$ and let $Y=(y_1,y_2)$ an horizontal section. We know that $\P(E_{|U})\cong U\times\P^1$, and let $[1:y]=[y_1:y_2]$ be the coordinate on the fibers. Then:
\[d\begin{pmatrix}y_1\\y_2\end{pmatrix}+\begin{pmatrix}\omega_{1,1}&\omega_{1,2}\\\omega_{1,2}&\omega_{1,2} \end{pmatrix}\begin{pmatrix}y_1\\y_2\end{pmatrix}=0\;\;\;\implies \;\;\;dy=d\bigg(\frac{y_2}{y_1}\bigg)=\frac{dy_2\cdot y_1-dy_1\cdot y_2}{y_1^2}=\omega_{1,2}y^2+(\omega_{1,1}-\omega_{2,2})y-\omega_{2,1},\]
and the solutions of this Riccati differential equation define a Riccati foliation on $\P(E_{|U})$. \\As one can expect, the following result holds.
\begin{prop}
    Holomorphic gauge equivalent connections induce biholomorphic Riccati foliation. 
\end{prop}
\begin{proof}
    Section 1.4.3 of \cite{geompanv}.
\end{proof}
A fundamental consequence of this result and Proposition \ref{corO1} is that any meromorphic gauge class of Painlevé V connection can be represented as a Riccati foliation over $\F_1:=\P(\O(-1)\oplus\O)$, that is the so called first Hirzebruch surface. We recall that Hirzebruch surfaces $\F_k$ are smooth algebraic ruled surfaces birational to $\P^1\times\P^1$. Additionally, they can be interpreted as the compactification of the total space of the line bundle $\O(k)$, obtained by adding a so-called "section at infinity". Hirzebruch surfaces have a very rich birational geometry, the interested reader will find more details in \cite{Beauville_1996}. A geometrical property that we will use in the following is that in each Hirzebruch surface $\F_k$ there is a unique curve of self intersection $-k$: it corresponds to the line $\P(\O(k))\subseteq\P(\O\oplus\O(k))$, which is exactly the "section at infinity" we mentioned before to compactify the total space of $\O(k)$. \\
For the following, it is useful to understand how elementary transformations act on the Riccati foliation (see also Section 1.5 of \cite{Heu2019FlatRT}). Consider a point $x_0\in U\subseteq\P^1$ and $\{e_1,e_2\}$ a basis of $E_{|U}\cong U\times \C$ such that the elementary transformations have the coordinate expression as in the formula \ref{elemtransexpr}. If we consider a section $Y=(y_1,y_2)$, then
\[\mathrm{elm}^+_{x_0}(Y)=\begin{pmatrix}x-x_0&0\\0&1\end{pmatrix}\begin{pmatrix}y_1\\y_2\end{pmatrix}=\begin{pmatrix}(x-x_0)y_1\\y_2\end{pmatrix}.\]
In the projectivisation, we have hence that 
\[\mathrm{elm}^+_{x_0}[y_1:y_2]=[(x-x_0)y_1:y_2],\]
and we see that on the fiber above $x_0$, any local section is forced to pass through the point $[0:1]$. In fact, an elementary transformation is nothing more than a flip.
\begin{center}
    \includegraphics[width=12cm]{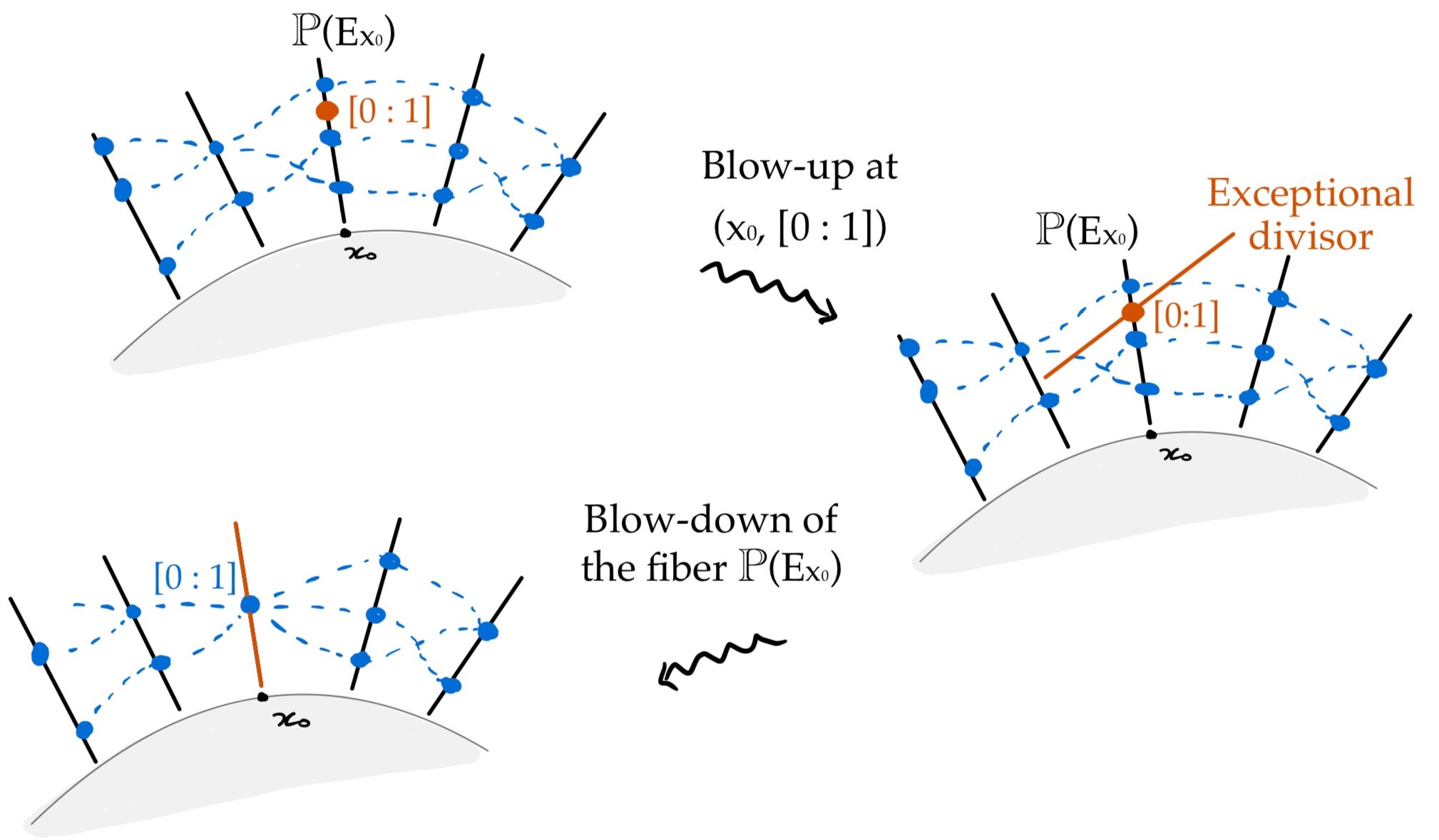}
\end{center}
Finally, as a consequence of Proposition \ref{propelmdeg}, we deduce that an elementary transformation is a birational morphism between $\F_k$ and $\F_{k\pm1}$. \\
Let us study in detail the case of PV type connections. Let $(\O(-1)\oplus\O, D, \nabla)$ be a PV type connection. The same reasoning used in \cite{Loray_2016}, Section 5, for a Painlevé VI type connection, shows that the Riccati foliation induced on $\F_1$ has exactly one tangency point with the section at infinity. Let us call $q$ the projection over $\P^1$ of that point. We can then apply an elementary transformation based in the tangency point, landing then in $\F_2$. All the leaves of the new Riccati foliation need then to pass through a certain point $(q,[1:\hat p])\in\F_2$, rising a radial singularity. The value of $\hat p$ is determined by the curvature of the leaf in $\F_1$ at its tangency point with $s_\infty$. 
\begin{center}
    \includegraphics[width=15cm]{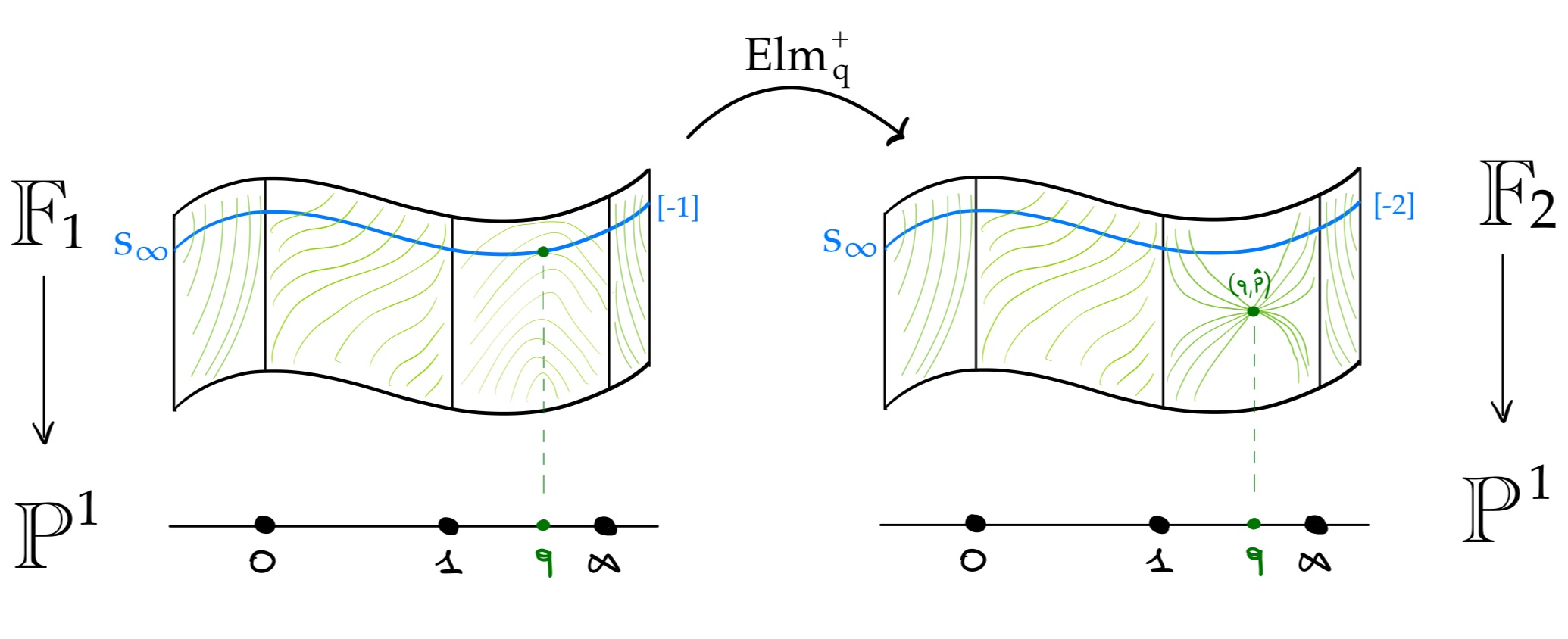}
    \captionof{figure}{}\label{elmq}
\end{center}

\subsection{Normal Form}\label{secnormform}
The goal of this section is to use the action of meromorphic gauge transformations to find a suitable representative in each isomorphism class of connections of PV type. This points out the link with the Riccati foliation, becoming the cornerstone for the construction of the moduli space. \\In a recent paper, Diarra and Loray \cite{Diarra} proved that we can choose a special representative by considering a connection in $E=\O\oplus\O(2)\to \P^1$ in companion form:
\begin{align*}
    \nabla_{|0}= & \,d+\Omega_0=\\&d+\begin{pmatrix}0&1\\0&t\end{pmatrix}\frac{dx}{(x-1)^2}+\begin{pmatrix}0&-1\\0&-\kappa_1\end{pmatrix}\frac{dx}{x-1}+\begin{pmatrix}0&1\\0&-\kappa_0\end{pmatrix}\frac{dx}{x}+\begin{pmatrix}0&0\\-\rho^{V}&0\end{pmatrix}xdx+\begin{pmatrix}0&0\\\hat p&-1\end{pmatrix}\frac{dx}{x-q}+\begin{pmatrix}0&0\\\hat K&0\end{pmatrix}dx,
\end{align*}
where 
\begin{itemize}
    \item $\kappa_0, \kappa_1, \kappa_\infty$ are generic residual spectral data, as explained in Section \ref{intro1};
    \item $\rho^V=\frac{(\kappa_0+\kappa_1-1)^2}{4}-\frac{\kappa_\infty^2}{4}$;
    \item $\hat K=\frac{\hat p^{2}}{q \left(q -1\right)^{2}}+\frac{\left(\mathit{\kappa_0} \left(q -1\right)^{2}+(\kappa_1-1)  q \left(q -1\right)-t q \right) \hat p}{q \left(q -1\right)^{2}}+\rho^V q +\frac{\hat p}{q -1}$;
    \item $t, \hat p$ and $q$ are three free parameters uniquely defining the normal form.
\end{itemize}
\begin{oss}
        In the literature, as in \cite{Ohyama_2006}, the parameter $\kappa_1$ can be found under the form as $\theta+1$. The advantage of using $\theta$ is its more direct link to the Painlevé fifth equation.
\end{oss}
\begin{oss}
    This normal form can also be obtained from the connection on $\O(-1)\oplus\O$ presented in \cite{Ohyama_2006} by changing the variables $(x,y,z)$ into $(x,q,-\frac{\hat p}{q(q-1)^2})$ and by applying the elementary transformation 
\[\begin{pmatrix}
    1&0\\0&\frac{-1}{x(x-1)^2}
\end{pmatrix},\]
landing on $\O\oplus\O(2)$.
\end{oss}
An apparent singularity in $x=q$ arises. It corresponds to the radial singularity of the Riccati foliation in $\F_2$ shown in Figure \ref{elmq}.
\begin{defn}
   A singularity for a meromorphic connection is said \textit{apparent} if it can be removed via a (meromorphic) gauge transformation.
\end{defn}
\begin{oss}
    Apparent singularities have nilpotent or diagonalizable  residual matrix with integer eigenvalues and trivial monodromy.
\end{oss}
We have now enough material for "decorating" a little bit more the picture of the Riccati foliation induced by the connection in normal form. Let us consider for instance the residual matrix of the logarithmic pole 0. Its diagonal form is
\[\begin{pmatrix}0&1\\0&-\kappa_0\end{pmatrix}=\begin{pmatrix}1&-\frac{1}{\kappa_0}\\0&1\end{pmatrix}\begin{pmatrix}0&0\\0&-\kappa_0\end{pmatrix}\begin{pmatrix}1&\frac{1}{\kappa_0}\\0&1\end{pmatrix},\]
pointing out the two eigenvalues $0$ and $-\kappa_0$, with the respective eigenvectors. We can then point out the two eigenvectors in the fiber over 0, represented by two points in $\F_2$, respectively $(0,[1:0])$ and $(0, [1:-\kappa_0])$, in red in Figure \ref{F2}. The same (with much more involved computations) can be applied for the other logarithmic pole at $\infty$. In the following of this section, the local coordinates on a trivialising set $U\subseteq\F_2$ are then called $(x,y)\in\P^1\times\P^1$.
\begin{oss}
    There is a geometrical interpretation of these points, making the link with the Riccati foliation more explicit. Indeed, considering the indiced Riccati differential equation in a small neighbourhood near $x=0$ yields
    \[dy=-\frac{1}{x}y^2+\frac{\kappa_0}{x}y \;\;\;\text{ and, for } x=0,\;\;\; 0=-y^2+\kappa_0y,\]
    whose solutions are precisely 0 and $-\kappa_0$. They correspond to the first terms of the Taylor expansions of the invariant curves of the foliation in a neighbourhood of $x=0$.
\end{oss}
For the double pole, the residual matrix alone is not sufficient: one has to consider the order two matrix as well. 
It becomes crucial to better understand the fiber over $x=1$. The local Riccati differential equation is in this case 
\[dy=\frac{x-2}{(x-1)^2}y^2+\frac{-\kappa_1(x-1) +t }{\left(x -1\right)^{2}}
y\;\;\;\text{ and for } x=1,\;\;\;0=(x-2)y^2+(-\kappa_1 x +t +\kappa_1)y\]
giving as solutions $y=0$ and $y=-\frac{t-\kappa_1(x-1)}{x-2}$. Their Taylor expansion of order 2 in $x=1$ are $0$ and 
\[t +\left(t-\kappa_1 \right) \left(x -1\right)+\mathrm{O}\! \left(\left(x -1\right)^{2}\right).\]
The special red points in the fiber over $x=1$ are then $(1,0)$ and $(1,t)\in \F_2$ (corresponding indeed to the eigenvectors of the second order matrix at $x=1$), and to get to the other terms of the Taylor expansion we have to blow up both those points getting to the following picture (in which we omitted the Riccati foliation) that perfectly represents the Okamoto divisor relative to the PV type connections, that is composed by the section at infinity, the singular fibers, and the exceptional divisors $E_{1,0}^\pm$, as shown in the picture below. 
\begin{center}
    \includegraphics[width=9cm]{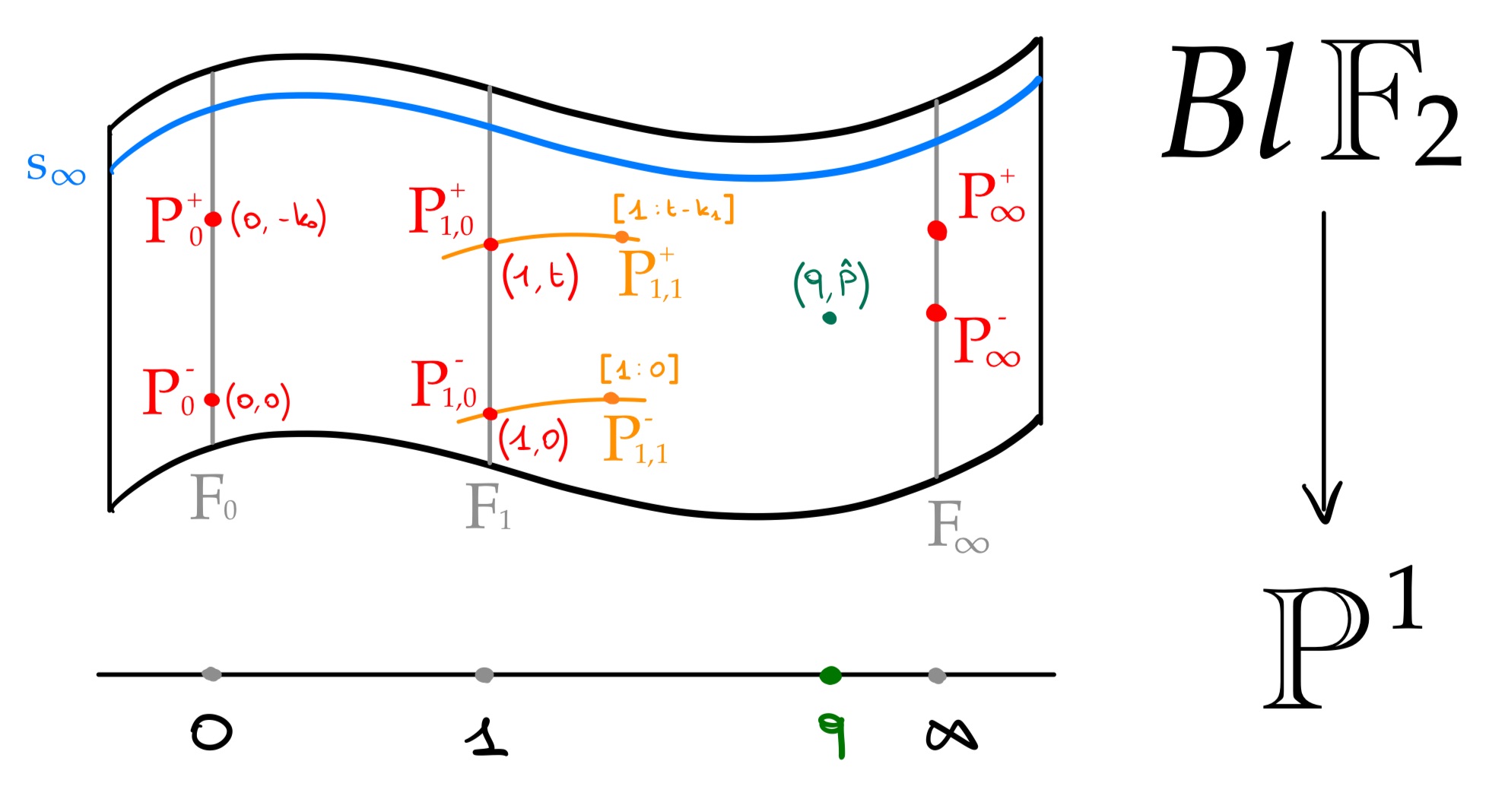}
     \captionof{figure}{}\label{F2}
\end{center}
\begin{oss}
     We refer to the Okamoto divisor as defined in \cite{saitotak}. We denote by $F_a$ the fiber above $x=a$ and with $E_{1,0}^\pm$ the two exceptional divisors respectively arising from the blowing up of $(1,0)$ and $(1,t)\in \F_2$, then the Okamoto divisor is 
     \[Ok^V:=2[s_\infty]+[F_0]+[F_1]+[F_\infty]+[E_{1,0}^+]+[E_{1,0}^-].\]
     Note that in \cite{saitotak}, each irreducible components of the Okamoto divisor has self intersection -2. This can be achieved if we blow-up the six points $P_0^\pm$, $P_{1,1}^\pm$ and $P_\infty^\pm$. The natural ambient space of the Okamoto divisor is then more likely the one appearing in Figure \ref{Okam}.
\end{oss}
We finally remark that also the point $(q,\hat p)$ has, as all the other special red points, the interpretation as a point $q\in \P^1$ and an eigenvector $(1,\hat p)$ of the residual matrix at $x=q$. In particular we stress that for each PV type connection we have a picture as the one above depending on the values of $t,q$ and $\hat p$.

\subsection{Confluence of Singularities}\label{Confl}

Let us conclude the present section with a preliminary study about confluence of singularities, that we will use for building the compactification of the moduli space of PV type connections. We wonder what happens when the apparent singularity $q$ converges into a pole. In other words: we wonder under which conditions, for $a=0,1,\infty$, the $\lim_{q\to a}\Omega_0$ does exist, giving a well defined connection matrix and how it looks like. \\
From the expression of the normal form, we see that for the coefficient
\[\hat K:=\frac{\hat p^{2}}{q \left(q -1\right)^{2}}+\frac{\left(\mathit{\kappa_0} \left(q -1\right)^{2}+(\kappa_1-1)  q \left(q -1\right)-t q \right) \hat p}{q \left(q -1\right)^{2}}+\rho^V q +\frac{\hat p}{q -1},\] the $\lim_{q\to a}\hat K$ is not well defined for any value of $a=0,1,\infty$.

\begin{thm}\label{confl}
    Given a PV type connection in normal form, the confluence of the apparent singularity into a pole $x=a$ produces a new well defined connection if, and only if, the point $(q,\hat p)\in\F_2$ follows the trajectory of the invariant curve at $x=a$ of the induced Riccati foliation up to the order $r$ of its Taylor expansion, where $r+1$ is the order of the pole $x=a$. Moreover, the connection depends on the $(r+1)$-th coefficient of the Taylor expansion of the trajectory.
\end{thm}
\begin{proof}
    We should explicitly compute $\lim_{q\to a}\Omega_0$, and we already noticed that only the coefficient $\hat K$ is divergent. \\Let us firstly study the case for $a=0$, that, we recall, is a logarithmic singularity. In Section \ref{secnormform} we proved that the Taylor expansion up to the order 0 of the invariant curves of the Riccati foliation at $x=0$ are
    \[\hat p =0 \;\;\;\text{ and }\;\;\; \hat p = -\kappa_0.\]
    We recall that, since $x=0$ is a logarithmic pole, the Taylor expansion of order 0 should be sufficient. \\We can hence make a substitution into $\hat K$ by replacing $\hat p$ with $\alpha q+ \beta$. Our goal is then to show that the limit $q\to0$ is well defined if, and only if, $\beta=0,-\kappa_0$. It then holds that 
    \[\mathrm{Res}_{q=0}\hat K= \beta(\beta+\kappa_0)\]
    that is indeed zero if, and only if, $\beta=0, -\kappa_0$, as desired. Moreover a straightforward computation shows that the resulting connection depends on $\alpha$.\\
    The computations for $a=\infty$ (that is the other logarithmic pole) are the same. \\For the pole of order two the procedure is similar, but we have to take in account two terms of the Taylor expansion of the invariant curve: 
    \[\hat p = 0 \;\;\;\text{ and }\;\;\;\hat p = -\kappa_1(q-1)+t\]
    We can hence make a substitution into $\hat K$ by replacing $\hat p$ with $\alpha (q-1)^2+ \beta(q-1)+\gamma$ and, as before, one can show that the limit $q\to1$ is well defined if, and only if, $(\beta,\gamma)=(0,0),(-\kappa_1,t)$. It then holds that 
    \[\mathrm{Res}_{q=1}\hat K= \gamma^{2}+\left(t +2 \beta -\theta -1\right) \gamma +\beta  t
    \;\;\;\;\;\text{ and }\;\;\;\;\;\mathrm{Res}_{q=1}(q-1)\hat K= \gamma(\gamma-t).\]
    Moreover a straightforward computation shows that the resulting connection depends on $\alpha$, concluding the proof.
\end{proof}
We are now ready to state the following result, which will be useful in Section \ref{seccomp} to understand the whole 3-dimensional moduli space and its compactification.
\begin{cor}\label{okmodsp}
    The moduli space of PV type connections for a fixed value of $t\in \C\setminus\{0\}$ (called the Okamoto moduli space) is given by the complement of the Okamoto divisor inside $\mathrm{Bl}\F_2$, that is the blowup of $\F_2$ in the eight points $P_i^\pm$ and $P_{1,j}^\pm$ of Figure \ref{F2}.
\end{cor}
\begin{proof}[Sketch of proof]
    The normal form tells us that for $q\neq0,1,\infty$ and $\hat p \neq \infty$, the point $(q,\hat p)$ uniquely define a PV type connection. The last part of Theorem \ref{confl} shows that for $q=0,1,\infty$, there exist PV type connections lying exactly on those exceptional divisors, in green in the following picture. For a complete proof in the Painlevé VI type case see Theorem 2.7.6 of \cite{geompanv}.
\end{proof}
\begin{center}
    \includegraphics[width=8cm]{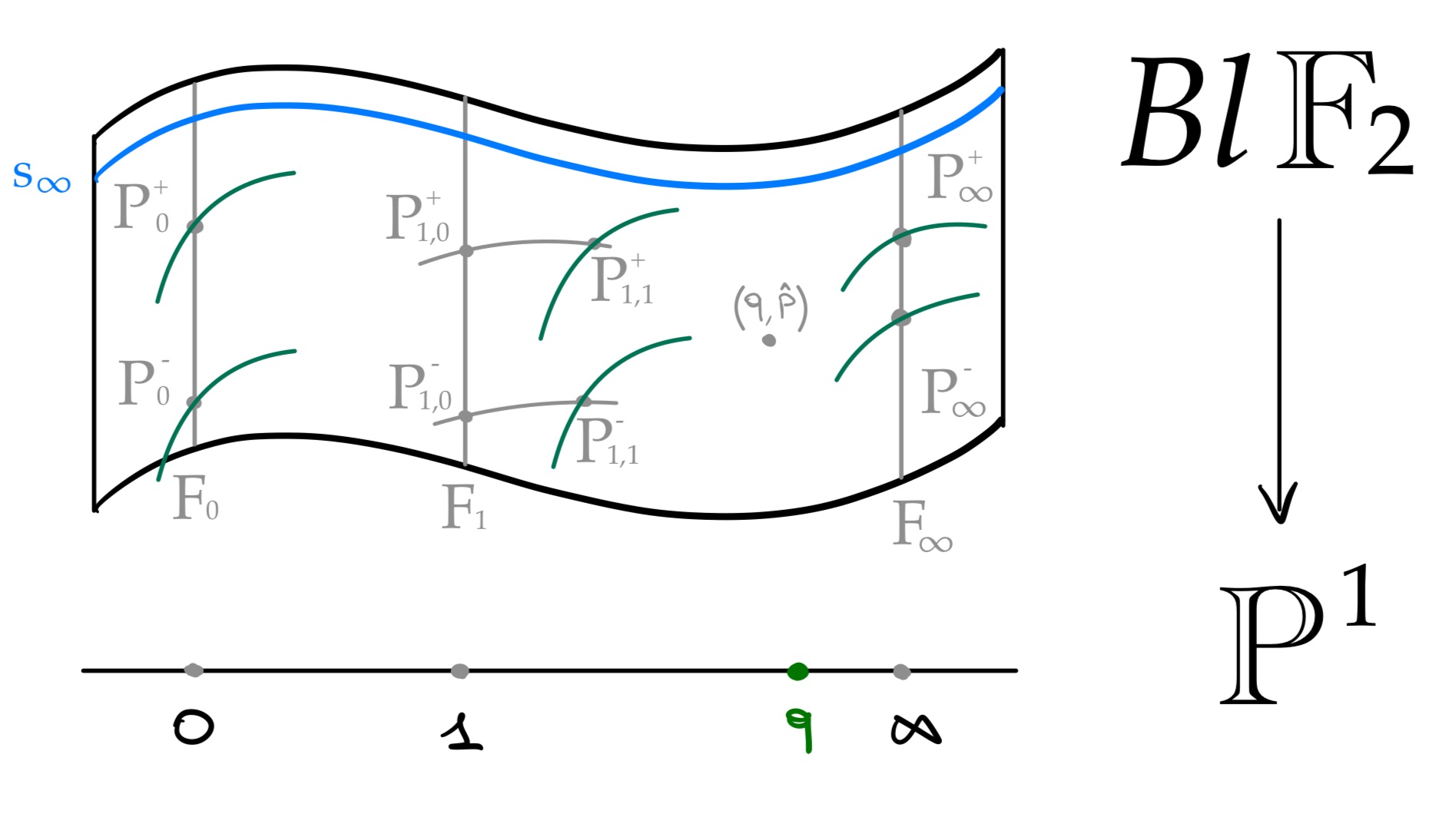}
    \captionof{figure}{}\label{Okam}
\end{center}

\section{Moduli Space of connections of PV type}\label{S1}
The goal of this section is to study $\Conn$: the moduli space of connections of PV type with fixed residual spectral data $\Theta$ up to gauge equivalence. As mentioned in Section \ref{PVconn}, in a recent paper, Diarra and Loray \cite{Diarra} proved that we can uniquely choose a representative by considering a connection in normal form on the vector bundle $E=\O\oplus\O(2)\to \P^1$.\\
The three free parameters 
$t,q,\hat p$ are then good candidates for being the coordinates in a dense open set of the moduli space of connections of PV type. We can not hope that they parametrise the entire moduli space, since we have seen in the previous section that inside $\Conn$ we should allow connections with $q$ equal to 0, 1 or $\infty$. In other words, we can see $\Conn$ as a family of Okamoto moduli spaces as presented in Corollary \ref{okmodsp} indexed by $t\in\C\setminus\{0\}$. The goal is to understand what happens for $t=0,\infty$.\\We already explored the geometric interpretation and the role of $q$ and $\hat p$ in Section \ref{Confl}. The geometric interpretation of the eigenvalue $t$ seems more mysterious, and it can be understood by studying how it behaves under the action of a change of coordinates in the base $\P^1$. 
\begin{lemma}\label{tangent}
    Let $f\in\mathrm{Aut}(\P^1)$ be a Moebius transformation and $(E, D, \nabla)$ a connection of PV type in normal form. Then $f^*\nabla$ is holomorphically gauge equivalent to $(E, D, \nabla)$, with the same residual spectral data and the eigenvalue $t$ is modified by $t\mapsto Df(1)\cdot t$.
\end{lemma}
\begin{proof}
   Let us apply the variable change $X=f(x)$ and denote by $a=f(1)$. A straightforward calculation shows that, for $i\in\{0,\infty,q\}$:
   \[\mathrm{Res}_{X=f(i)}\big(f^*\Omega\big)=\mathrm{Res}_{x=i}\big(\Omega\big) \;\;\;\;\text{ and }\;\;\;\;\mathrm{Res}_{X=a}\big((X-a)\cdot f^*\Omega\big)=Df(1)\cdot\begin{pmatrix}0&1\\0&t\end{pmatrix},\]
   and we are done.
\end{proof}
In particular we deduce that $t$ transforms like a vector tangent to $1\in\P^1$, that is $t\in T_1\P^1\setminus\{0\}$. 
In order to build a moduli space we want to define a variety in which these coordinates can live. We recall that
\begin{itemize}
    \item $q\in \P^1$, the base of the vector bundle,
    \item $\hat p\in\P(E_q)$, the projectivisation of the fiber above $q$,
    \item $t\in T_1\P^1$, the tangent space to $\P^1$ in 1.
\end{itemize}
In conclusion, we see $t$ and $q$ as elements of the quasi projective variety 
\[\mathcal M:=\Big\{q\in\P^1\setminus\{0,1,\infty\};\;\;t\in T_1\P^1\setminus \{0\}\Big\}\cong\P^1\times\P^1\setminus(Q^0\cup Q^1\cup Q^\infty\cup A_0\cup A_\infty),\]
where $Q^i:=\{q=i\}$ and $A_j=\{t=j\},$ as shown in the picture below, representing the universal curve above $\M$. 
\begin{center}
    \includegraphics[width= 10cm]{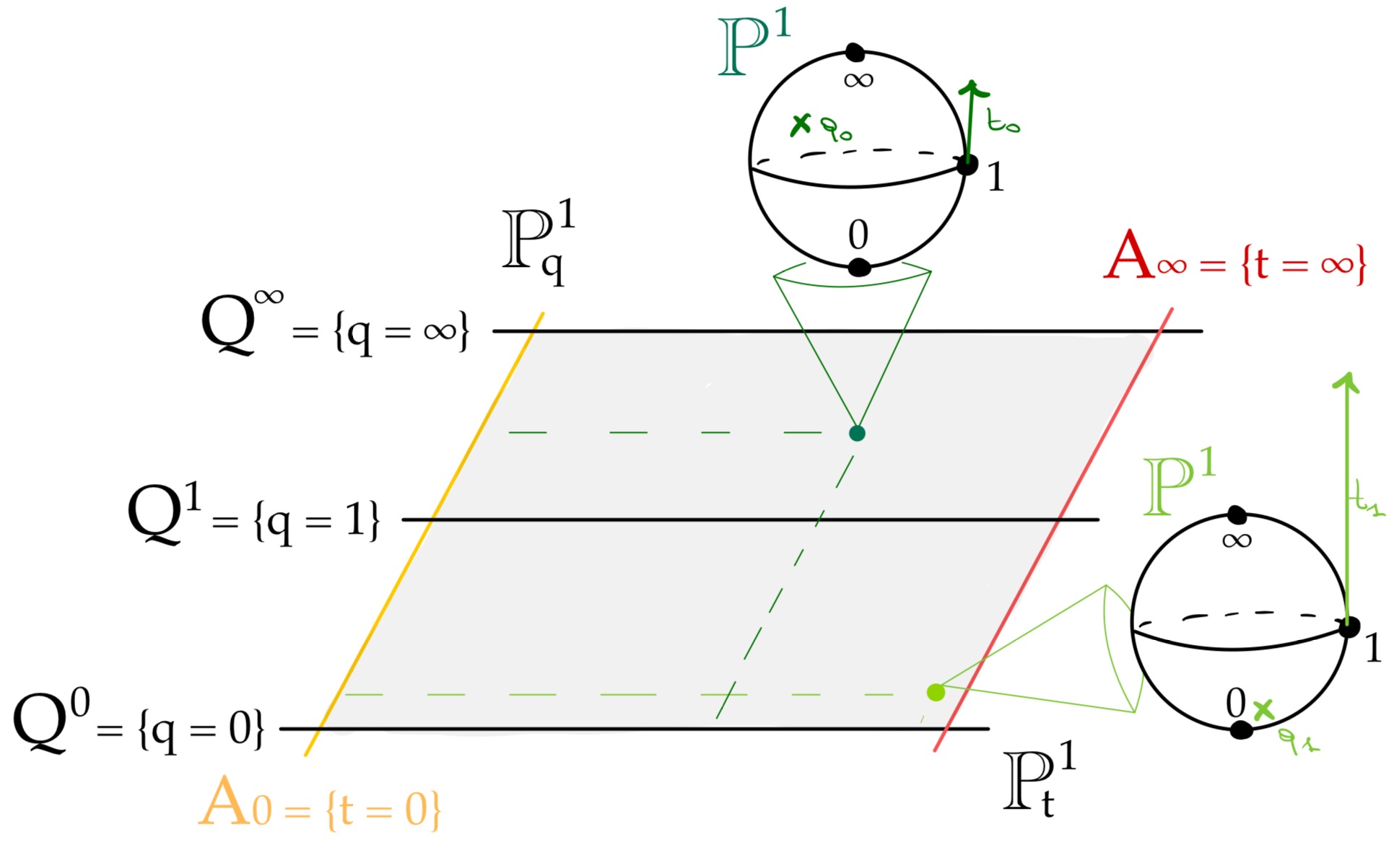}
    \captionof{figure}{}\label{pic1}
\end{center}
A PV connection is hence the datum of a such irregular curve and an extra parameter $\hat p\in \C$, as we will better precise later.

\subsection{Irregular Curves}\label{secirrcurv}
The space $\mathcal M$ can be seen as a natural degeneration of the moduli space $\mathcal{M}_{0,5}$ of configurations of five points in $\P^1$. In our context, instead of the fifth point, we have a tangent vector. Before giving the main definition, let us recall the definition of a jet, as we will use it in the following.
\begin{defn}
    A $r$-jet of coordinates over $\P^1$ is a point of $J^r(\R_0, \P^1)$. In particular, given a smooth curve $\gamma\colon \R_0\to \P^1$ and a coordinate of $\P^1$ around $\gamma(0)$, we set $j^r\gamma=(\gamma(0), \gamma^{(1)}(0), \dots, \gamma^{(r)}(0))$, where $\gamma^{(i)}$ is the $i$-th derivative of $\gamma$. 
\end{defn}
This lead us to give the following definition, that has been introduced in \cite{irrcurv} and \cite{boalchirrcurv}.
\begin{defn}
    An irregular curve $(\mathcal C, D, J)$ is the datum of a complex curve $\mathcal C$, an effective divisor $D=\sum_{i\in I} a_i[p_i]\in \mathrm{Div}(\mathcal C)$ and a collection of jets $J=\big\{j^{a_i-1}\gamma_i\big\}_{i\in I}$, where $\gamma_i\colon \R_0\to \P^1$ is a smooth curve such that $\gamma_i(0)=p_i$. Sometimes we will omit $D$ or $J$ in the notation.
\end{defn}
\begin{oss}
    The quasi projective variety $\mathcal M$ is the moduli space of irregular curves $(\P^1, [a]+2[b]+[c]+[d], J)$, for distinct $a,b,c,d\in \P^1$ and with $J=\{j^0(a)=a, j^1(b)=(b,t), j^0(c)=c,j^0(d)=d\}$. 
\end{oss}
\begin{oss}\label{ossautirr}
    An automorphisms $\phi\in\mathrm{Aut}(\P^1)$ acts on $(\P^1,D, J)$ as $(\phi(\P^1), \phi_*D, \phi_*J)$. If we call $p=\gamma(0)$, we have that $\phi_*\big(j^k\gamma\big)=\big(\phi(p), D\phi(p)\left(\gamma^{(1)}(0)\right),\dots,D^k\phi(p)\left(\gamma^{(k)}(0)\right)\big)$.
\end{oss}
The datum of a generic gauge equivalence class of a connection of type PV is equivalent to the data of a rational irregular curve in $\mathcal M$ and an additional parameter $\hat p\in \C$:
\[[\nabla] \iff \Bigg\{\begin{matrix}\big(\P^1,[0]+2[1]+[\infty]+[q], J\big)\\J=\{0,(1,t), \infty, q\}\end{matrix}\;\;\text{ and }\;\hat p\in \C\;\Bigg\}\iff (t,q,\hat p)\in\M\times\C.\]
\begin{oss}
    We talk about \textit{generic} connections since we exclude the cases $q=0,1,\infty$ and $t=0,\infty$.
\end{oss}
We can then prove the following.
\begin{prop}\label{openmod}
    The moduli space of connections of PV type $\Conn$ contains as a dense open set the trivial line bundle $\mathcal M\times \C$, where $\mathcal M$ is the moduli space of configuration of four points and a tangent vector on $\P^1$.
\end{prop}
\begin{proof}
    The normal form appearing in \cite{Diarra} shows that $\Conn\supseteq\mathcal M\times \C$. Moreover, what is missing are such connections in which $q=0,1,\infty$. Finally, thanks to Lemma \ref{tangent}, we can give $\mathcal M$ the interpretation as the moduli space of configuration of four points and a tangent vector on $\P^1$, concluding the proof.
\end{proof}

\section{Compactified Moduli Space of connections of PV type}\label{compsect}
If we want to compactify $\Conn$, we should before understand how to compactify $\M$. We proceed by the following steps: 
\begin{itemize}
	\item[$\bullet$] We are inspired by Kapranov's article \cite{kapranov} to compactify $\mathcal M$ and we study the irregular stable nodal curves lying on $Q^i$ and $A_j$, in a Deligne-Mumford's flavour. 
	\item[$\bullet$] We describe the connections on such curves. That will be crucial to establish a suitable set of coordinates in $\overline \Conn$.
	\item[$\bullet$] We understand how the trivial bundle $\mathcal M\times \C$ extends over the compactified basis $\overline{\mathcal{M}}$.
\end{itemize}

\subsection{Compactification of $\mathcal M$}\label{subs:compM}
We take inspiration by the compactification that M. Kapranov \cite{kapranov} described for $\mathcal M_{0,5}$, the moduli space of configurations of five points in $\P^1$. His main idea is to establish an isomorphism between $\M_{0,5}$ and the moduli space of smooth marked conics in $\P^2$ passing through four given points in general position (with the marking that is outside the four given points). This isomorphism associates to any such marked conic an element in $\M_{0,5}$, that is indeed a $\P^1$ with five punctures. The advantage of this isomorphism is that a compactification of the moduli space of smooth marked conics is much easier to get, and it is then straightforward to deduce the desired compactification $\overline{\M_{0,5}}$.
\begin{oss}
    We can see $\P^1$ with five punctures as a rational irregular curve $(\P^1 ,[a]+[b]+[c]+[d]+[e])$. Note that the divisor has degree five.
\end{oss}
We adapt the reasoning to $\mathcal M$, that is the moduli space of rational irregular curves of the form $(\P^1, [0]+2[1]+[\infty]+[q], J)$ up to Moebius transformations. \\
We then consider smooth marked conics in $\P^2$ passing through three given points $A,B,C\in \P^2$ in general position and with a prescribed tangent vector $l\in T_A\P^2\setminus \{0\}$. Since we want all these data to be "in general position", we ask the tangent direction $l$ not to point to $B$ or $C$. Up to automorphisms of $\P^2$, we can suppose these data to be
\[A=[1:0:0], \;\;\;\;B=[0:1:0],\;\;\;\;C=[0:0:1],\;\;\;\; l=\overrightarrow{AD} \;\text{ for }\; D=[1:1:1].\]
From now on, we denote by $\mathfrak{C}:=\P^2\setminus(\overline{AB}\cup\overline{AC}\cup \overline{BC}\cup \overline{AD})$, where $\overline{AB}$ denotes the unique line in $\P^2$ passing through $A$ and $B$. 
\begin{lemma}
    The quasi projective variety $\mathfrak{C}$ is the moduli space of smooth marked conics passing through three fixed point $A,B,C\in \P^2$ in general position, and with prescribed tangent direction $l\in T_A\P^2$ not pointing $B$ or $C$.
\end{lemma}
\begin{proof}
    It is well known that the moduli space of conics in $\P^2$ is isomorphic to $\P^5$ by taking the projectivisation of $\C^6$, the space of the coefficients of the polynomial equation. In particular, a conic has five degrees of freedom. Imposing the four requested conditions, the requirement that the conic pass through an additional point $P\in \P^2\setminus\{A,B,C\}$ uniquely determines the conic. Moreover, the conic is smooth if, and only if, $P$ is chosen in $\mathfrak C$. Finally, let $P$ and $P'$ be two different points defining the same conic $\mathcal C$, then the marked conics $(\mathcal C;A,B,C,P, l)$ and $(\mathcal C;A,B,C,P',l)$ are not isomorphic since $\dim \mathrm{Aut}(\P^2)=8$ and hence the only automorphism of $\P^2$ fixing pointwise $\{A,B,C, l\}$ is the identity.
\end{proof}
In the picture below, on the left, we show a conic $\mathcal C$ represented by a point $P\in\mathfrak C$, while, on the right, it is shown that picking a $P\in \P^2\setminus\mathfrak C$ gives rise to a singular conics, represented in yellow and blue.
\begin{center}
    \includegraphics[width=9cm]{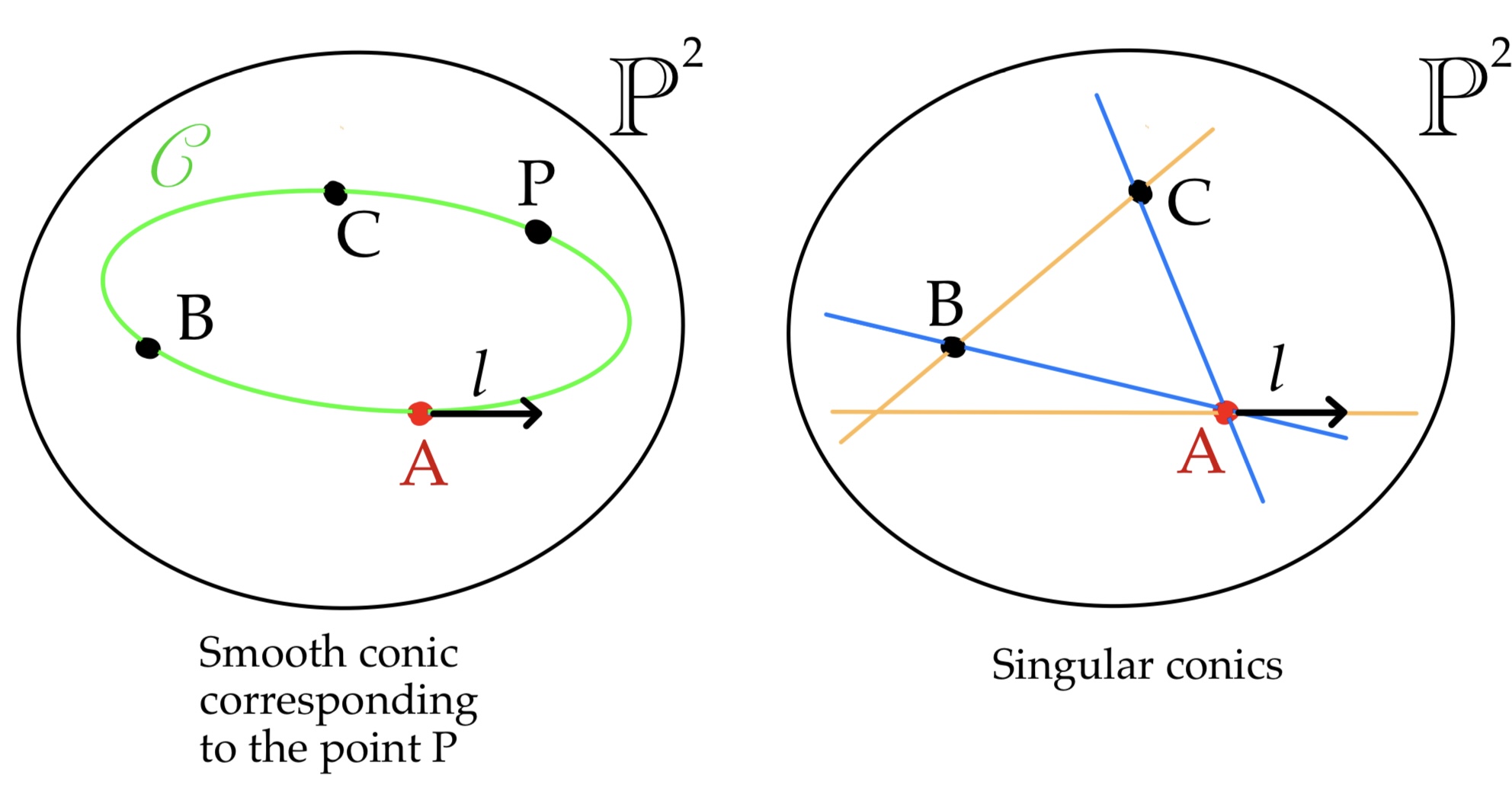}
    \captionof{figure}{}\label{pic2}
\end{center}
Now, let us construct the explicit isomorphism $\mathcal M\cong \mathfrak C$. Once we have it, we can build the compactification $\overline{\mathfrak C}$ and deduce $\overline{\mathcal M}$. We recall that a point $(t,q)\in\M$ represents a class of rational irregular curves, that we would like to identify with a conic in $\mathfrak C$. To do so, we have to find a suitable way to associate with each point $P\in \mathfrak C$ the parameters $(t,q)$.\\ Let us call $[X_0:X_1:X_2]$ the coordinates on $\P^2$. Imposing that the conics pass through $A, B, C$ and be tangent to $l$, we find that points of $\mathfrak C$ represent conics of the form $X_1X_2+a(X_0X_1-X_0X_2)=0$ for $a\in \C.$ 
We want to consider these curves with the parametrisation $\P^1\to\P^2$ that identifies $0,1,\infty\in \P^1$ with respectively $B,A,C\in \P^2$:
\[\begin{matrix}
    \P^1 & \dashrightarrow & \P^2\\ 
    [s_0:s_1] & \mapsto &[s_0s_1:as_1(s_0-s_1):as_0(s_0-s_1)]\\ 
    0=[0:1] & \mapsto &[0:1:0]=B\\ 
    1=[1:1] & \mapsto &[1:0:0]=A\\ 
    \infty=[1:0] & \mapsto &[0:0:1]=C
\end{matrix}.\] 
Imposing that the conic passes through the point $P=[X_0^0:X_1^0:X_2^0]$, we get
\[a=\frac{X^0_2X_1^0}{X_2^0X_0^0-X_1^0X_0^0}.\]
In fact, the point $P$ plays the role of $q$ in the rational irregular curve, and in order to get its expression, we have to find the preimage of $P$ under the parametrisation. \\Since $\mathfrak C\cap \{X_0\neq0\}=\mathfrak C$, we can 
work in the affine charts $V_1:=\{s_1\neq0\}\subseteq \P^1$ and $U_0:=\{X_0\neq0\}$, with coordinates $x_i=X_i/X_0$. The parametrisation becomes
\[\begin{matrix}\gamma :&V_1 & \to & U_0\\ &\frac{s_0}{s_1}=s & \mapsto & \Big(\frac{a(s-1)}{s},a(s-1)\Big)\end{matrix}.\]
The preimage of $P$ is then 
\[q=\gamma^{-1}(P)=\frac{x_2^0}{x_1^0}.\] 
The role of $t$ is, of course, played by the tangent vector of the conics in $A$. We already know that it has the same direction as $\overline{AD}$. A computation shows that 
\[\gamma'(1)=(a,a)\;\;\;\text{ and then } \;\;\;t=\frac{1}{a}=\frac{x_2^0-x_1^0}{x_2^0x_1^0}\]
We get then the following birational map
\[\begin{matrix}
    \phi\colon&\P^2&\dashrightarrow &\P^1\times\P^1\\
    
    &[X_0:X_1:X_2]&\mapsto&\big([X_2X_0-X_1X_0:X_1X_2],[X_2:X_1]\big)\\&(x_1,x_2)\in U_0&\mapsto&(t,q)=\Big(\frac{x_2-x_1}{x_2x_1},\frac{x_2}{x_1}\Big)\in V_1\times V_1
\end{matrix}\]
that is not defined in $A,B,C$. We prove that
\begin{lemma}\label{lemma1}
    The restricted map $\phi_{|\mathfrak C}\colon \mathfrak C\to \M$ is an isomorphism.
\end{lemma}
\begin{proof}
We refer to Figure \ref{pic1} for the notations, recalling that 
\[\mathfrak{C}\cong\P^2\setminus(\overline{AB}\cup\overline{AC}\cup \overline{BC}\cup \overline{AD})\;\;\;\text{and}\;\;\;\mathcal M\cong\P^1\times\P^1\setminus(Q^0\cup Q^1\cup Q^\infty\cup A_0\cup A_\infty).\]
We have that $\phi_{|\mathfrak C}$ is everywhere well defined and injective, and hence it induce an isomorphism into its image. It suffice to show $\phi(\mathfrak C)=\M$. First of all we remark that $Q^0, Q^1, Q^\infty$ and $A_\infty$ are not in the image of $\phi$. Moreover $A_0=\phi(\overline{BC})$ and $\overline{BC}\not\subset\mathfrak C$. The other lines not appearing in $\mathfrak C$ are such that 
\[\phi(\overline{AB})=([1:0],[0:1])=(\infty,0)=A_\infty\cap Q^0, \]
\[\phi(\overline{AC})=([1:0],[1:0])=(\infty,\infty)=A_\infty\cap Q^\infty,\;\;\;\;\;\;\phi(\overline{AD})=([0:1],[1:1])=(0,1)=A_0\cap Q^1.\]
And then $\phi(\mathfrak C)=\M$ as desired.
\end{proof}
We construct $\overline{\mathfrak C}$ by identifying the limit conics that lie on the boundary components.
\\\textbf{First Step:}\\
We can easily add to $\mathfrak{C}$ the points that lie on $(\overline{AB}\cup\overline{AC}\cup\overline{BC}\cup\overline{AD})\setminus\{A,B,C\}$ by taking into account singular conics, as shown on the right part of Figure \ref{pic2}.
\\\textbf{Second Step:}\\
Since we are less comfortable working with tangent vectors, we can blow-up the point $A\in \P^2$, and call $L$ the point in the exceptional divisor $E_A$ corresponding to the direction $l$, as shown in the following picture. Notice that, now, the condition of tangency has been translated by the condition of "passing through $L$". 
\begin{center}
    \includegraphics[width=5cm]{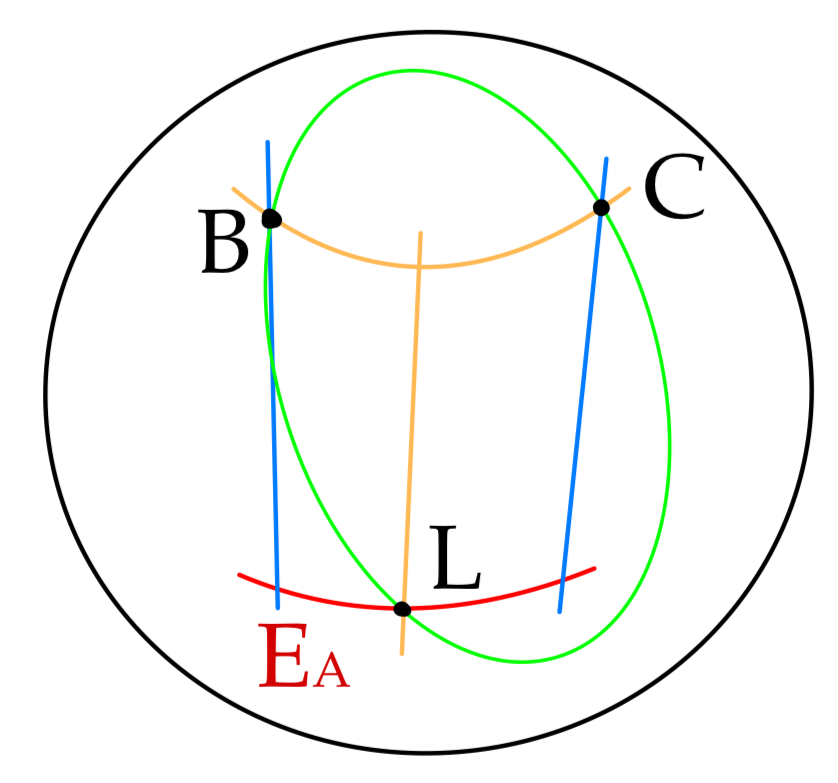}
\end{center}
\textbf{Third Step:}\\
In order to distinguish conics that pass through $B,C$ and $L$ by the slope of incidence, let us blow-up these three points, denoting respectively by $E_B$, $E_C$ and $E_L$ the exceptional divisors. In the following picture we denote into brackets the self-intersection number of any shown curve.
\begin{center}
    \includegraphics[width=5cm]{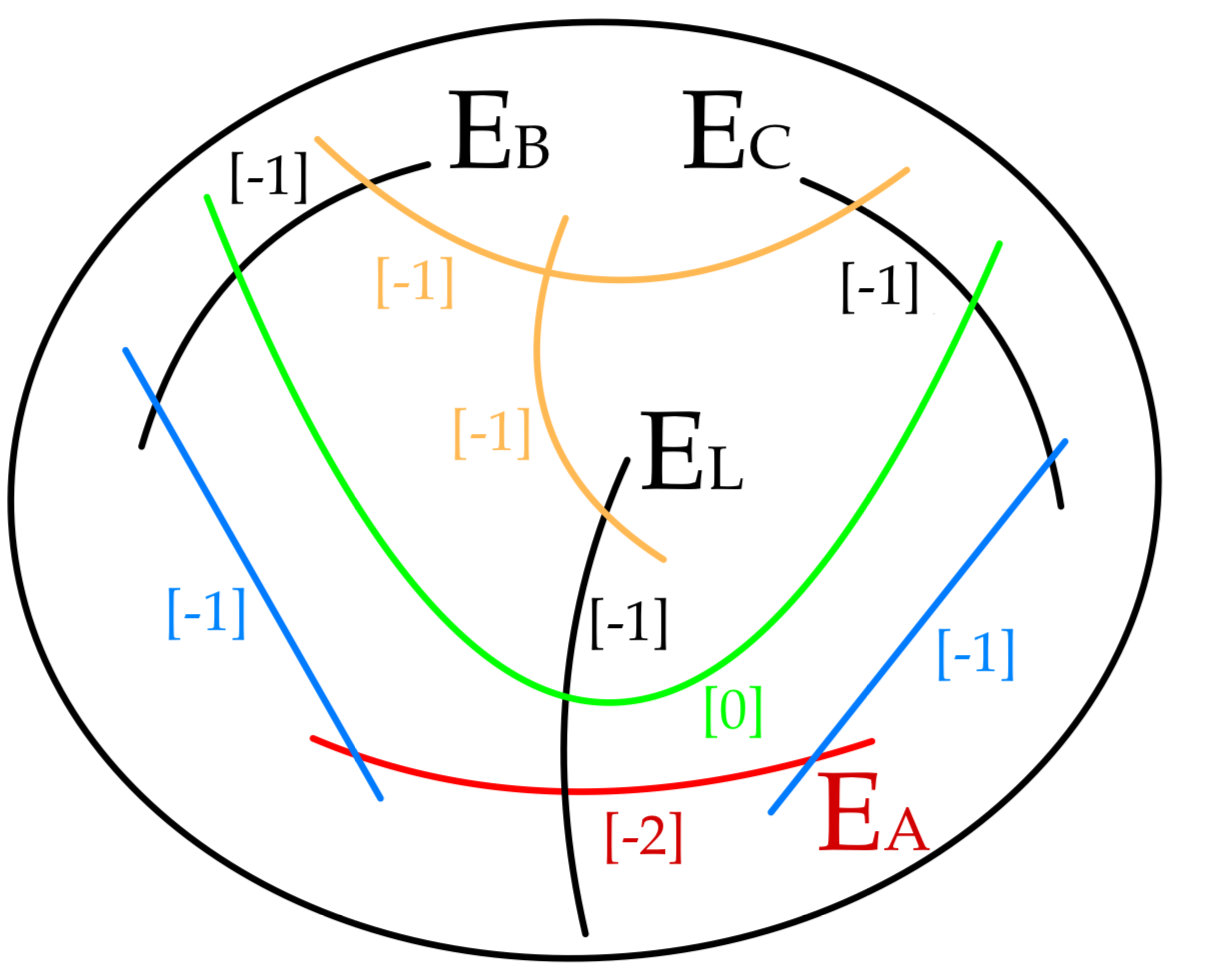}
    \captionof{figure}{}\label{p2blow}
\end{center}
\begin{lemma}\label{modC}
    Let us denote by $X$ the blowup of $\P^2$ shown in Picture \ref{p2blow} and by $P_1 ,P_2$ the intersection points between $E_A$ and the blue lines. There is a bijection between $(X\setminus E_A)\cup \{P_1,P_2\}$ and the set of marked conics in $\P^2$ passing through $A,B,C$ and with prescribed tangent $l$, with the exception that $P_1$ and $P_2$ define isomorphic marked conics. In particular $\overline{\mathfrak C}$ is isomorphic to the (singular) space obtained by the contraction of the $(-2)$-curve $E_A$ inside $X$.
\end{lemma}
\begin{proof}
    Recalling that the exceptional divisor $E_A$ represents the incidence slopes into $A$, the first part follows by construction since we only consider conics incident in $A$ with the prescribed slope $l$. Finally, the points $P_1$ and $P_2$ induce isomorphic marked conics, since we can exchange the two branches via an automorphism. As a consequence, we can contract $E_A$ to $P_1$ giving the thesis.
\end{proof}
We postpone to the next section the study of the stable nodal curve corresponding to points lying in the boundary components of $\overline{\mathfrak C}$.

We are now ready to build the compactification $\overline{\mathcal M}$.
\begin{lemma}\label{lemma3}
    The blowup of $\P^1\times\P^1$ at the points $(0,1),(\infty,0)$ and $(\infty, \infty)$ give rise to a smooth variety $\widetilde{\M}$ that is isomorphic to the variety $X$ appearing in Lemma \ref{modC}. Moreover, the isomorphism $\phi$ of Lemma \ref{lemma1} lifts to an isomorphism $\widetilde \phi\colon \widetilde X\to \widetilde\M$.
\end{lemma}
\begin{proof}
    The following picture shows the intersection numbers of the curves arising in $\widetilde{\M}$.
    \begin{center}
    \includegraphics[width=6cm]{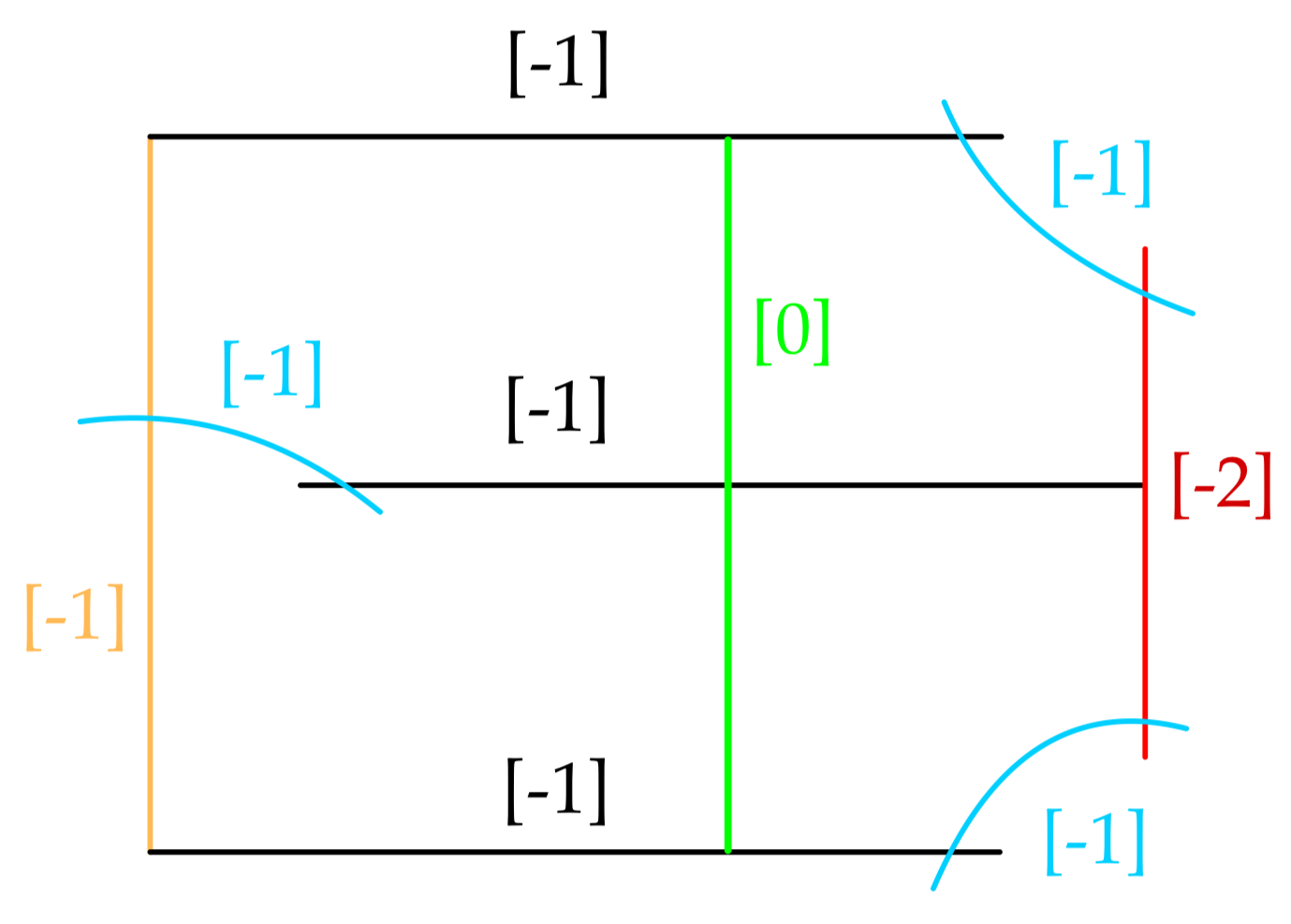}
    \captionof{figure}{}\label{p1blow}
    \end{center}
    Both $\widetilde{\M}$ and $X$ are then isomorphic to a weak Del Pezzo surface of degree 5 and since they have the same number of $(-1)$ and $(-2)$-curves, they are then isomorphic, since the moduli space of such weak Del Pezzo surfaces is just a point as shown in \cite{Martin2020WeakDP}. \\
    We have now to show that $\phi$ lifts to an isomorphism between $X$ and $\widetilde \M$ as shown in Figures \ref{p2blow} and \ref{p1blow}. We already noticed in Lemma \ref{lemma1} that $\phi(\overline{BC})=A_0$. Let us call $E_A, E_B$ and $E_C$ the exceptional divisors of the points $A,B$ and $C$ respectively. It holds that
    \begin{align*}\phi(E_A)&=\lim_{X_1\to0}\phi([1:X_1:\alpha X_1])=\lim_{X_1\to 0}([\alpha X_1-X_1:\alpha X_1^2],[\alpha X_1:X_1])=\\&=\lim_{X_1\to 0}([\alpha -1:\alpha X_1],[\alpha :1])=([1:0],[\alpha:1])=A_\infty.\end{align*}
    With a similar computation we find that $\phi(E_B)=Q^0$ and $\phi(E_C)=Q^\infty$.\\
    In Lemma \ref{lemma1}, we have also shown that $\phi(\overline{AB})=(\infty,0)$, $\phi(\overline{AC})=(\infty,\infty)$ and $\phi(\overline{AD})=(0,1)$. We now blow-up these points in $\P^1\times\P^1$ and we show that the image of the lines in $\P^2$ fit perfectly the exceptional divisors. We have that
    \[\mathrm{Bl}_{(\infty,\infty)}(\P^1\times\P^1)=\{([t_1:t_2],[q_1:q_2],[R_1:R_2])\in(\P^1)^3\;|\;t_2R_1=q_2R_2\}.\]
    Therefore $\phi(\overline{AC})=\phi([1:0:R])=([1:0],[1:0],[R:1])=B^\infty_\infty$. And the same holds for the two other points. This conclude the proof.
\end{proof}
This implies the following result.
\begin{thm}\label{compbasis}
    The compactification of the moduli space of stable nodal curves $\overline{\mathcal M}$ is given by the contraction of the $(-2)$-curve $A_\infty$ inside $\widetilde{\M}$.   
\end{thm}
\begin{proof}
    We proved in Lemma \ref{lemma1} that $\M\cong \mathfrak C$. Then we deduced in Lemma \ref{modC} the compactification $\overline{\mathfrak C}$. Following the same path for $\M$, we get to the variety $\widetilde \M$ that we showed being isomorphic to $X$ in Lemma \ref{lemma3}, identifying $A_\infty$ with $E_A$, that is the curve we contracted in $X$ to get to $\overline{\mathfrak C}$. 
\end{proof}
We point out the analogy of the curves in $\widetilde{\M}$ and $X$. In order to make the comparison clear, let us give a name to some of the curves in $\widetilde{\M}$ that are needed later on in this article.
\begin{center}
    \includegraphics[width=10cm]{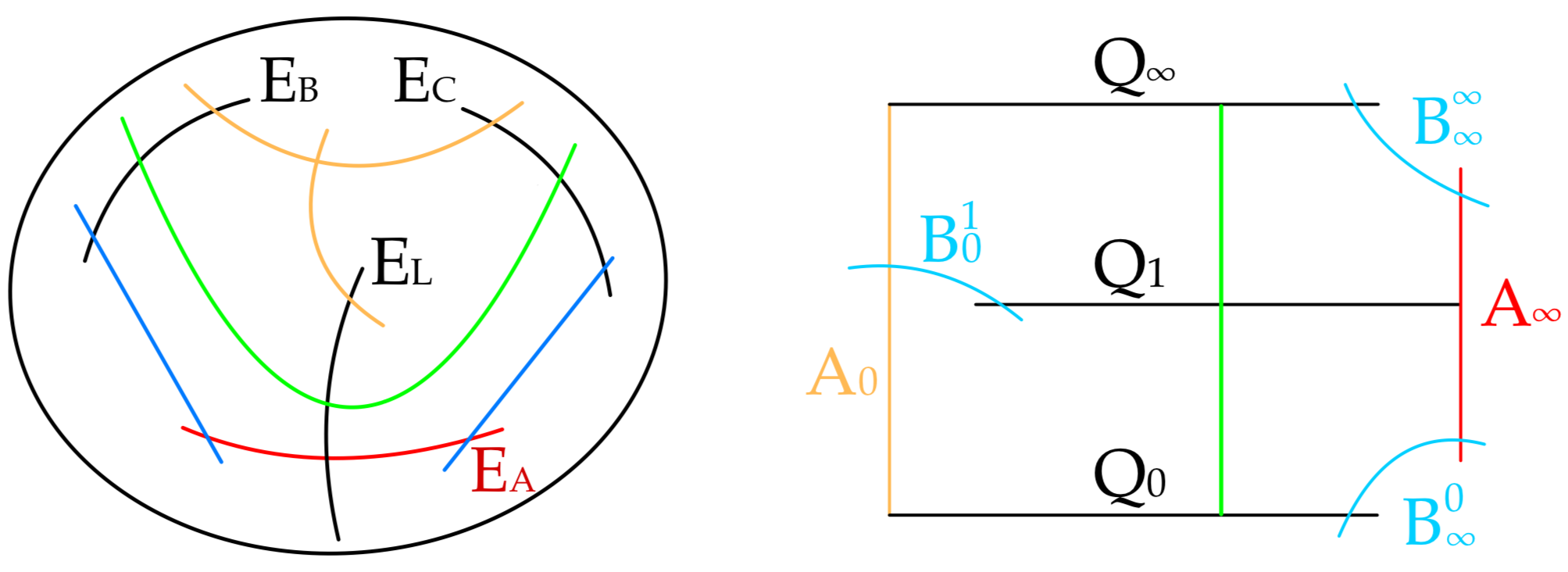}
    \captionof{figure}{}\label{boundfig}
    \end{center}
\[\begin{matrix}
    E_B &\leadsto&Q^\infty&&&&\{\text{Blue singular conic}\}&\leadsto&B_\infty^0\cup B_\infty^\infty\\
    E_L &\leadsto&Q^1&&\text{and}&&E_A &\leadsto&A_\infty\\
    E_C &\leadsto&Q^0&&&&\{\text{Yellow singular conic}\}&\leadsto&A_0\cup B_0^1
\end{matrix}
\]
Where $B_j^i$ is the exceptional divisor relative to the point $(t,q)=(j,i)$. We end this section with the following result that describes the singularities arising in $\overline{\mathcal M}$. 
\begin{prop}\label{Msing}
    The space $\overline{\mathcal M}$ presents strictly canonical singularities of $A_1$-type.
\end{prop}
\begin{proof}
First of all we recall that $\overline{\mathcal M}$ is given by the contraction of the $(-2)$-curve $A_\infty$ of $\widetilde{\M}$.\\ If we set $(t,q)$ the local coordinates of $\widetilde\M$ centred at $A_0\cap Q^0$, the volume form $\frac{dt}{t^2}\wedge\frac{dq}{q^2}$ gives the canonical divisor
\[K_{\widetilde{\M}}\equiv-2(Q^0+A_0)-B^1_0-B^0_\infty+B^\infty_\infty.\]
In particular it holds that 
\[(K_{\widetilde{\M}} \cdot A_\infty)=0,\]
and the singularity arising in $\overline{\mathcal M}$ is hence canonical. To show that are strictly canonical, we can compute the discrepancy: let $\pi\colon \widetilde{\M}\to \overline{\mathcal M}$ the contraction. Then
\[K_{\widetilde{\M}}\equiv\pi^*K_{\overline{\mathcal M}}+\alpha A_\infty, \]
and computing the intersection product with $A_\infty$ we find that
\[0=(K_{\widetilde{\M}} \cdot A_\infty)=(\pi^*K_{\overline{\mathcal M}}\;\cdot A_\infty)-2\alpha,\]
and $(\pi^*K_{\overline{\mathcal M}}\;\cdot A_\infty)=0$ since $A_\infty$ has been contracted to a point. This implies that $\alpha=0$, as desired.  
\end{proof}

\subsection{Irregular Stable Nodal Curves}\label{secstabnod}
We would like to extend the universal curve over $\M$ (Figure \ref{pic1}) to the boundary components of $\overline \M$. To do so, we introduce the Deligne and Mumford's notion of stable nodal curve used in the compactification of the moduli spaces $\mathcal M_{g,n}$. This study will be relevant to understand how the connections on the boundary components of $\overline\Conn$ are made, since they are connections over stable nodal curves, as we will see in the next section. 
\begin{defn}
    An irregular curve $(\mathcal C, D, J)$, as defined in Section \ref{secirrcurv}, is \emph{stable} if $\mathrm{Aut}(\mathcal C, D, J)$ is finite, where each element of $\mathrm{Aut}(\mathcal C, D, J)$ must fix pointwise $J$. 
\end{defn}
\begin{defn}
    A complex curve $\mathcal C$ is called a \emph{nodal curve} if it has only nodes as singularities.
\end{defn}
\begin{fact}
    Let $\mathcal C$ be a nodal curve with a nodal singularity at the point $p$ and let $\mathcal C_1, \mathcal C_2$ be the two smooth components. If we do not allow the branches to be exchanged, then $\mathrm{Aut}(\mathcal C)\cong\mathrm{Aut}(\mathcal C_1; p)\times\mathrm{Aut}(\mathcal C_2;p)$. Notice that the action $\mathrm{Aut}(\mathcal C)\curvearrowright \mathcal C$ decomposes into the two separated actions $\mathrm{Aut}(\mathcal C_1)\curvearrowright \mathcal C_1$ and $\mathrm{Aut}(\mathcal C_2)\curvearrowright \mathcal C_2$.
\end{fact}
\begin{defn}
    An \emph{irregular nodal curve} is a triple $(\mathcal C, D, J)$, where $\mathcal C$ is a nodal curve with smooth components $\mathcal C_1, \dots, \mathcal C_n$, $D$ is a collection of divisors $D_1, \dots, D_n$ and $J$ a collection $J_1, \dots, J_n$ of jets such that, for each $i=1,\dots,n$, the pair $(\mathcal C_i, D_i, J_i)$ is a smooth irregular curve.
\end{defn}
\begin{defn}
    The automorphism group $\mathrm{Aut}(\mathcal C, D, J)$ must fix pointwise the jets in each $J_i$ and the nodes. The irregular nodal curve $(\mathcal C, D, J)$ is said \emph{stable} if its automorphism group is finite.
\end{defn}
Of course, this is a generalisation of the stable nodal curves described in \cite{DelMun} and \cite{Arbarello2011}.\\
As shown in Figure \ref{boundfig}, the boundary components of $\overline\M$ have an interpretation as confluence of points. For example $Q^0$ is the boundary components that allows curves in which $q$ has converged into 0. Anyway, we do not want the curves over $Q^0$ just being a $\P^1$ with four markings instead of five: we want to keep the information of the confluence. We recall moreover that we are considering $\overline \M$ as a moduli space of punctured curves up to Moebius transformations. This helps us since confluences are not invariant under the action of $\mathrm{Aut}(\P^1)$.
\begin{es}
    Let us consider a Moebius transformation $f$ that sends $q$ to $\infty$, and that fixes 0 and $t$. We remark that such transformation does exist since $\mathrm{Aut}(\P^1)$ is 3-transitive. Some of our markings moved, but $\P^1$ and $f(\P^1)$ are of course in the same equivalence class. We should compare the confluence $q\to0$ in both cases. In the original $\P^1$ the point $q$ is just collapsing into 0, but in $f(\P^1)$ the points $f(1)$ and $f(\infty)$ will collapse together into 1. Roughly speaking, the confluence $q\to0$ is Moebius equivalent to the confluence $1\to\infty$, and, since $\P^1$ and $f(\P^1)$ are equivalent, we want to keep trace of both configurations using stable nodal curves.
\end{es} 
In those kind of curves we will always have two smooth components that will encode the two different behaviours of a confluence, as we will see just below.

\medskip
\textbf{Curves over $A_0$.}
The boundary $A_0$ represents those curves with $t=0$, but it turns out that setting $t=0$ is not the only possible representation. We recall indeed that we are considering curves up to isomorphism and, in particular, we can apply a Moebius transformation $f_{A_0}$ that fixes 1 and such that $Df_{A_0}(1)\cdot t= 1$. We can then study the limit for $t\to0$ and look at what the resulting curve looks like. One such Moebius transformation is 
\begin{equation}\label{fA0}
    f_{A_0}(x)=\frac{(x+1)t}{(t-2)x+t+2}.
\end{equation}
It holds that 
\[\lim_{t\to0}f_{A_0}(0)=\lim_{t\to0}f_{A_0}(q)=\lim_{t\to0}f_{A_0}(\infty)=0.\]
\begin{center}
    \includegraphics[width=10cm]{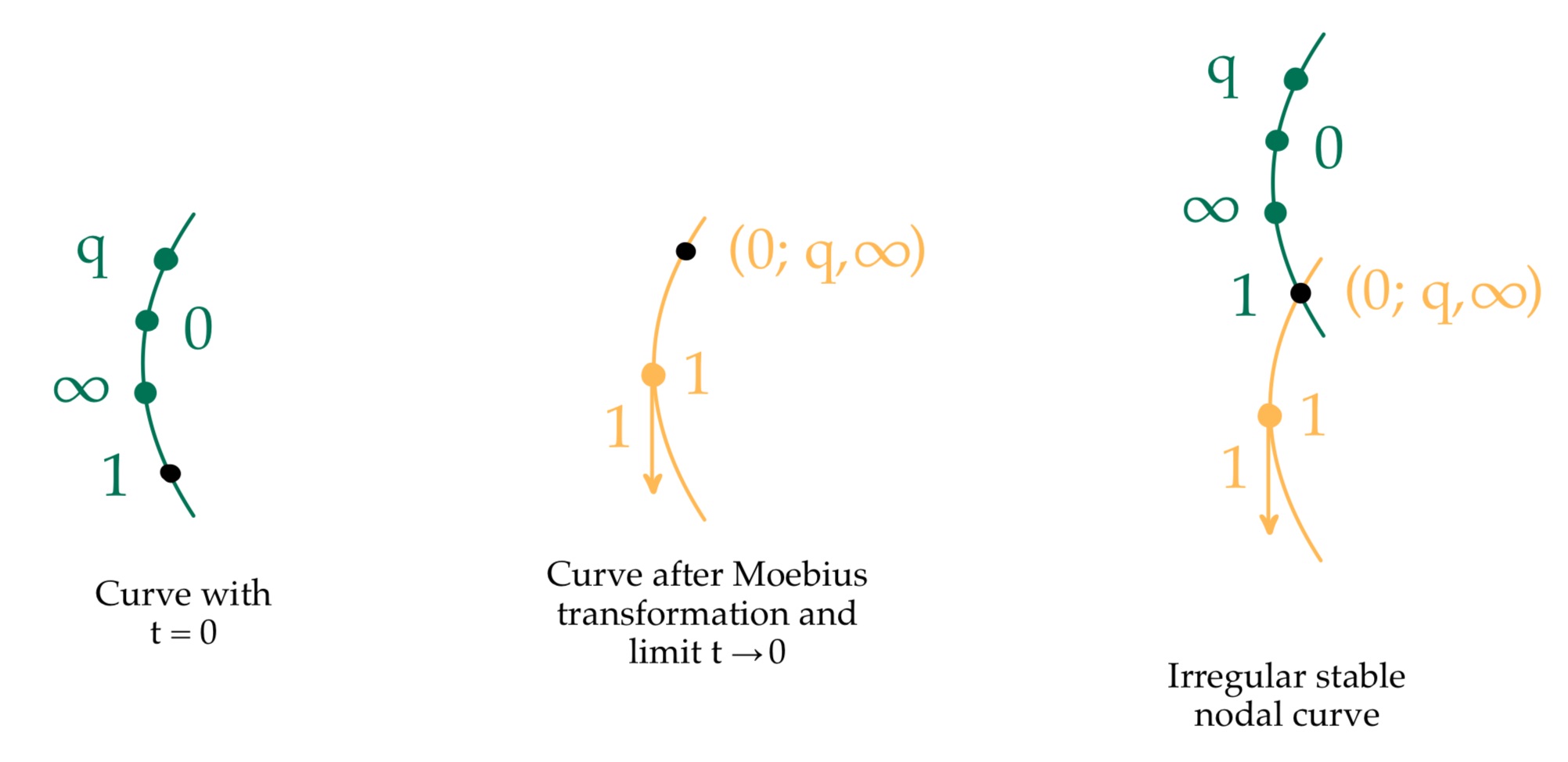}
    \captionof{figure}{}\label{A0}
\end{center}

\textbf{Curves over $Q^0$.}
The boundary $Q^0$ represents those curves with $q=0$. We can apply a Moebius transformation that fixes 0 and sends $q\to \infty$. We can then study the limit for $q\to0$ and look at what the resulting curve looks like. One such Moebius transformation is 
\begin{equation}\label{fQ0}
f_{Q^0}(x)=\frac{x}{x-q}.
\end{equation}
By a straightforward computation it holds that 
\[\lim_{q\to0}f_{Q^0}(1)=\lim_{q\to0}f_{Q^0}(\infty)=1,\;\;\;\;\text{ and }\;\;\;\;\lim_{q\to0}Df_{Q^0}(1)\cdot t=1.\]
\begin{center}
    \includegraphics[width=9cm]{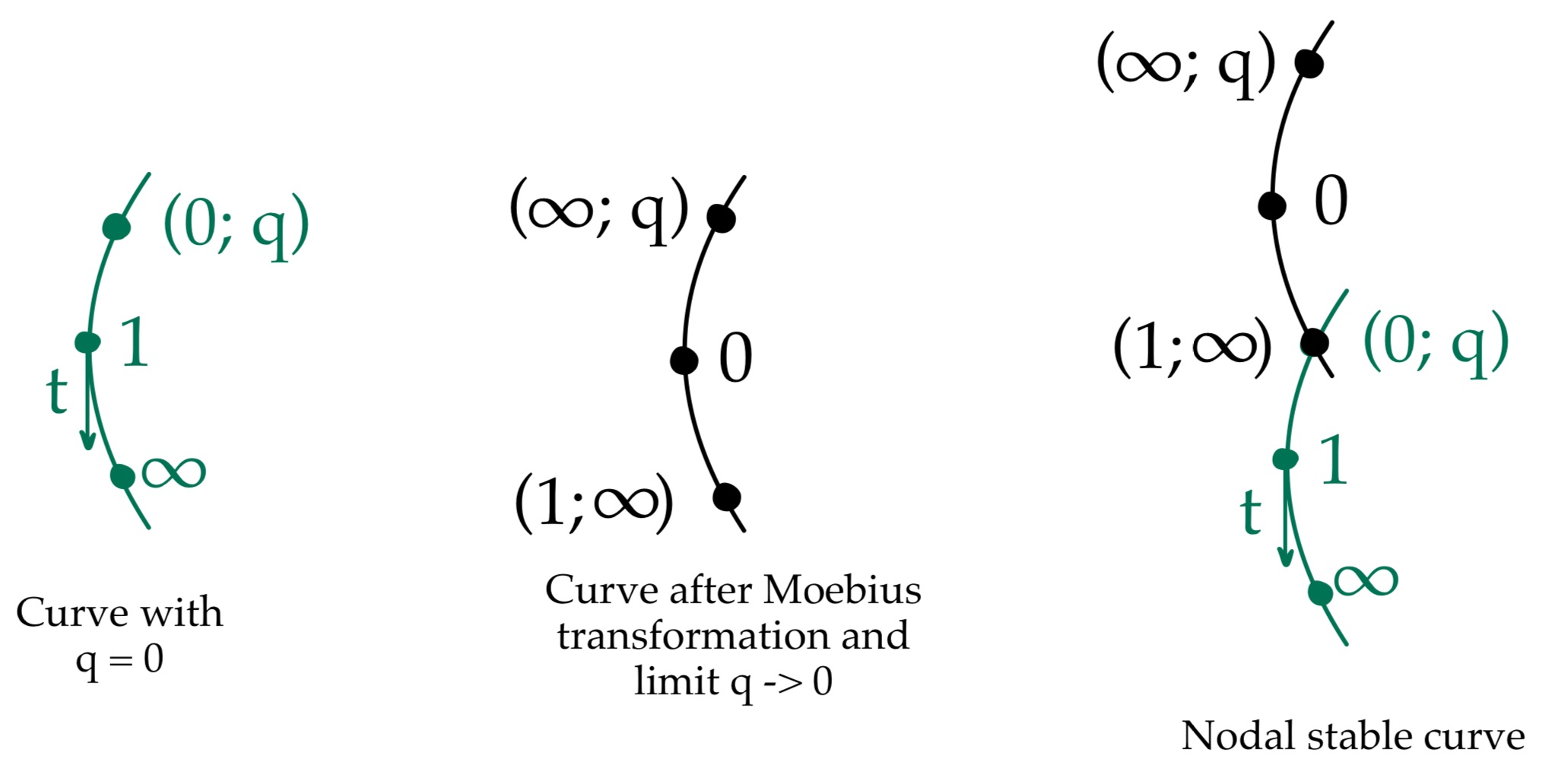}
\end{center}
\textbf{Curves over $Q^1$.}
The boundary $Q^1$ represents those curves with $q=1$. We can apply a Moebius transformation that fixes 1 and sends $q\to 0$. We can then study the limit for $q\to1$ and look at what the resulting curve looks like. One such Moebius transformation is 
\begin{equation}\label{fQ1}
f_{Q^1}(x)=\frac{t(x-q)}{(q - 1 + t)x + (-qt - q + 1)}.
\end{equation}
By a straightforward computation it holds that
\[\lim_{q\to1}f_{Q^1}(0)=\lim_{q\to1}f_{Q^1}(\infty)=1,\;\;\;\;\text{ and }\;\;\;\;\lim_{q\to1}Df_{Q^1}(1)\cdot t=1.\]

\textbf{Curves over $Q^\infty$.}
The boundary $Q^\infty$ represents those curves with $q=\infty$. We can apply a Moebius transformation that fixes $\infty$ and sends $q\to 0$. We can then study the limit for $q\to\infty$ and look at what the resulting curve looks like. One such Moebius transformation is 
\[f_{Q^\infty}(x)=-\frac{x}{q} + 1.\]
By a straightforward computation it holds that
\[\lim_{q\to\infty}f_{Q^\infty}(0)=\lim_{q\to\infty}f_{Q^\infty}(1)=1,\;\;\;\;\text{ and }\;\;\;\;\lim_{q\to\infty}Df_{Q^\infty}(1)\cdot t=0.\]
The following pictures show how the nodal curves described so far look like. The black and yellow smooth components represent the limit curves obtained after the application of the Moebius transformation. 
\begin{center}
    \includegraphics[width=7cm]{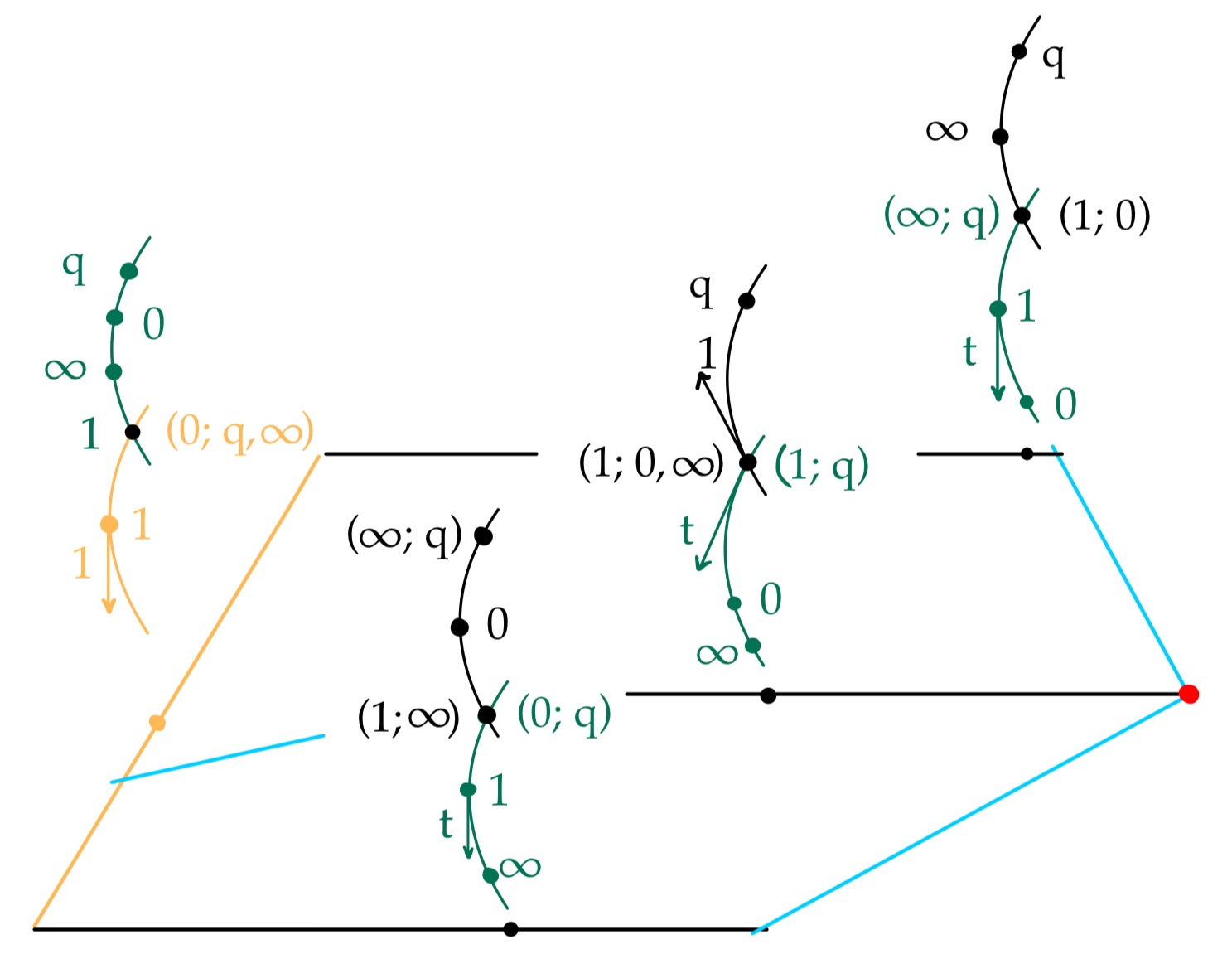}
    \captionof{figure}{}\label{stcurv}
\end{center}
\begin{oss}
    We see that, as expected, each curve shows only one parameter determining its position on the boundary component.
\end{oss}
\textbf{Curves over $B_0^1$.}
The boundary $B^1_0$ represents those curves with $t=0$ and $q=1$. We denote by $B=\frac{t}{q-1}$ the coordinate on this exceptional divisor. As for the other boundary components, one can either directly consider the curve with $t=0$ and $q=1$, either one can make the substitution $t=B(q-1)$ and apply a Moebius transformation that fixes $1$ and sends $q\to \infty$. We can then study the limit for $q\to1$ and look at what the resulting curve looks like. One such Moebius transformation is 
\[f_{B^1_0}(x)=\frac{\left(1-q \right) x}{x -q}.\]
By a straightforward computation it holds that
\[\lim_{q\to1}f_{B^1_0}(0)=0,\;\;\;\;\text{ and }\;\;\;\;\lim_{q\to1}\Big(Df_{B^1_0}(1)\cdot B(q-1)\Big)=B.\]
\begin{center}
    \includegraphics[width=7cm]{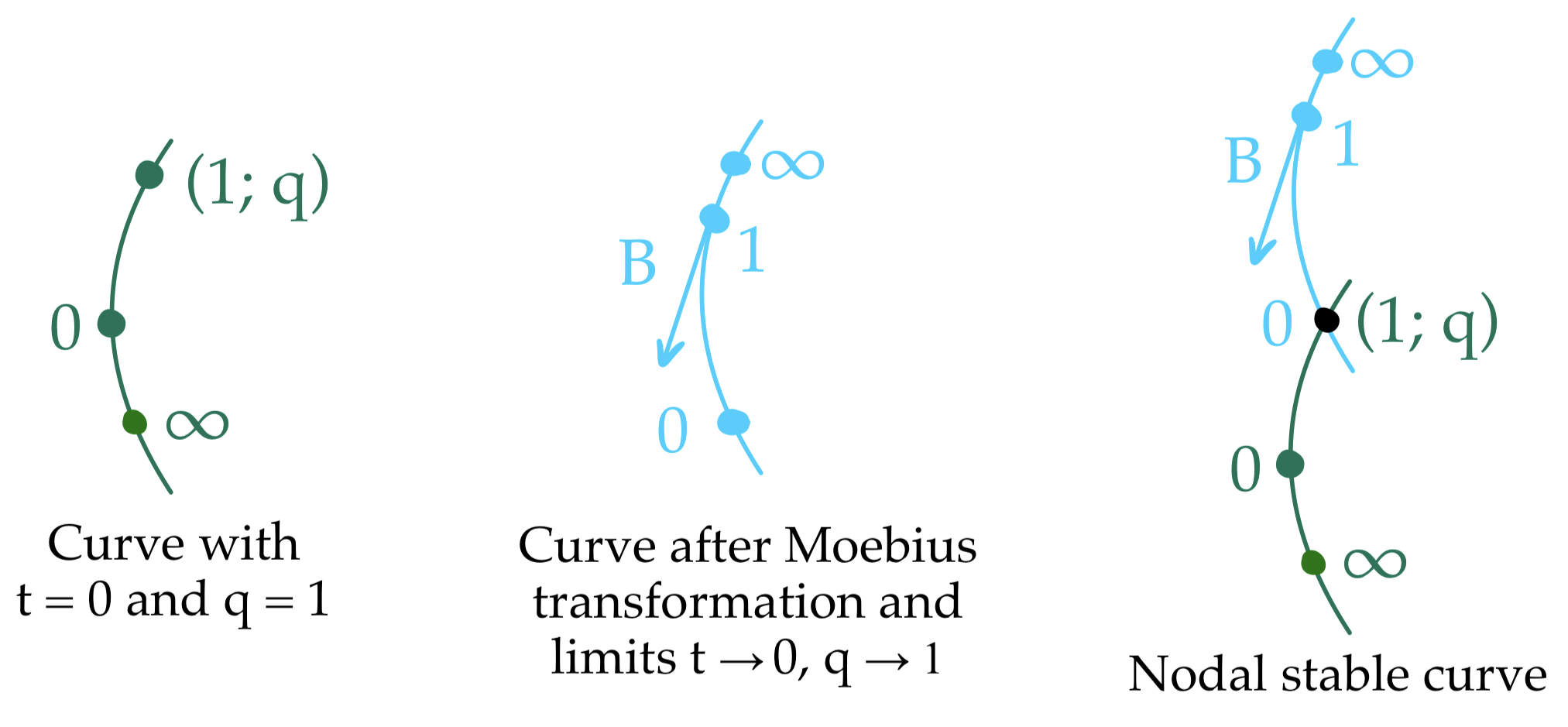}
    \captionof{figure}{}\label{B01}
\end{center}

\textbf{Curves over $B_\infty^0$.}
The boundary $B_\infty^0$ represents those curves with $t=\infty$ and $q=0$. We denote by $R=qt$ the coordinate on this exceptional divisor. Unlike the other times, here one can not directly consider the curve with $t=\infty$ and $q=0$, since we do not know what an infinite tangent vector is. We still have two possible options: either, via a Moebius transformation $f_{B_\infty^0}$, we fix 1, we normalise $t=1$ and we fix 0, either, with $g_{B_\infty^0}$, we fix 1, we normalise $t=1$, but we fix $\infty$. These two Moebius transformations are
\[f_{B_\infty^0}(x)=\frac{\left(t +1\right) x}{t  x +1}\;\;\;\text{ and }\;\;\;g_{B_\infty^0}(x)=\frac{x}{t}+\frac{t -1}{t}.
\]
By a straightforward computation, after, of course, substituting $q=\frac{R}{t}$, it holds that
\[\begin{matrix}
    \lim_{t\to\infty}f_{B_\infty^0}\Big(\frac{R}{t}\Big)=\frac{R}{R+1}&\text{ and }&\lim_{t\to\infty}Df_{B_\infty^0}(1)\cdot t=1\\&&\\
    \lim_{t\to\infty}g_{B_\infty^0}(0)=\lim_{t\to\infty}g_{B_\infty^0}\Big(\frac{R}{t}\Big)=1&\text{ and }&\lim_{t\to\infty}Dg_{B_\infty^0}(1)\cdot t=1
\end{matrix}\]

\textbf{Curves over $B_\infty^\infty$.}
The boundary $B_\infty^\infty$ represents those curves with $t=\infty$ and $q=\infty$. We denote by $S=\frac{t}{q}$ the coordinate on this exceptional divisor. As before, we have the same two possible options, and we can use the same functions
\[f_{B_\infty^0}(x)=\frac{\left(t +1\right) x}{t  x +1}\;\;\;\text{ and }\;\;\;g_{B_\infty^0}(x)=\frac{x}{t}+\frac{t -1}{t}.
\]
By a straightforward computation, after, of course, substituting $q=\frac{t}{S}$, it holds that
\[\lim_{t\to\infty}f_{B_\infty^0}(\infty)=\lim_{t\to\infty}f_{B_\infty^0}\Big(\frac{t}{S}\Big)=1,\;\;\;\;\text{ and }\;\;\;\;\lim_{t\to\infty}Df_{B_\infty^0}(1)\cdot t=1.\]
\[\lim_{t\to\infty}g_{B_\infty^0}\Big(\frac{t}{S}\Big)=\frac{S+1}{S},\;\;\;\;\text{ and }\;\;\;\;\lim_{t\to\infty}Dg_{B_\infty^0}(1)\cdot t=1.\]

\textbf{Curve over the point $\overline{A_\infty}$.}
The point $\overline{A_\infty}$ represents those curves with $t=\infty$ and any value of $q$. As in the other cases we have two possible options: we can fix 0, 1 and normalize $t=1$ or fix $\infty$, 1 and normalize $t=1$. We have then the two following Moebius transformations
\[f_{ A_\infty}(x)=\frac{t x}{(t-1)  x +1}\;\;\;\text{ and }\;\;\;g_{ A_\infty}(x)=\frac{x+t-1}{t}.
\]
In both cases
\[\lim_{t\to\infty}f_{ A_\infty}(x)=\lim_{t\to\infty}g_{ A_\infty}(x)=1,\]
and then all the non fixed points converge to 1. 
\begin{center}
    \includegraphics[width=9cm]{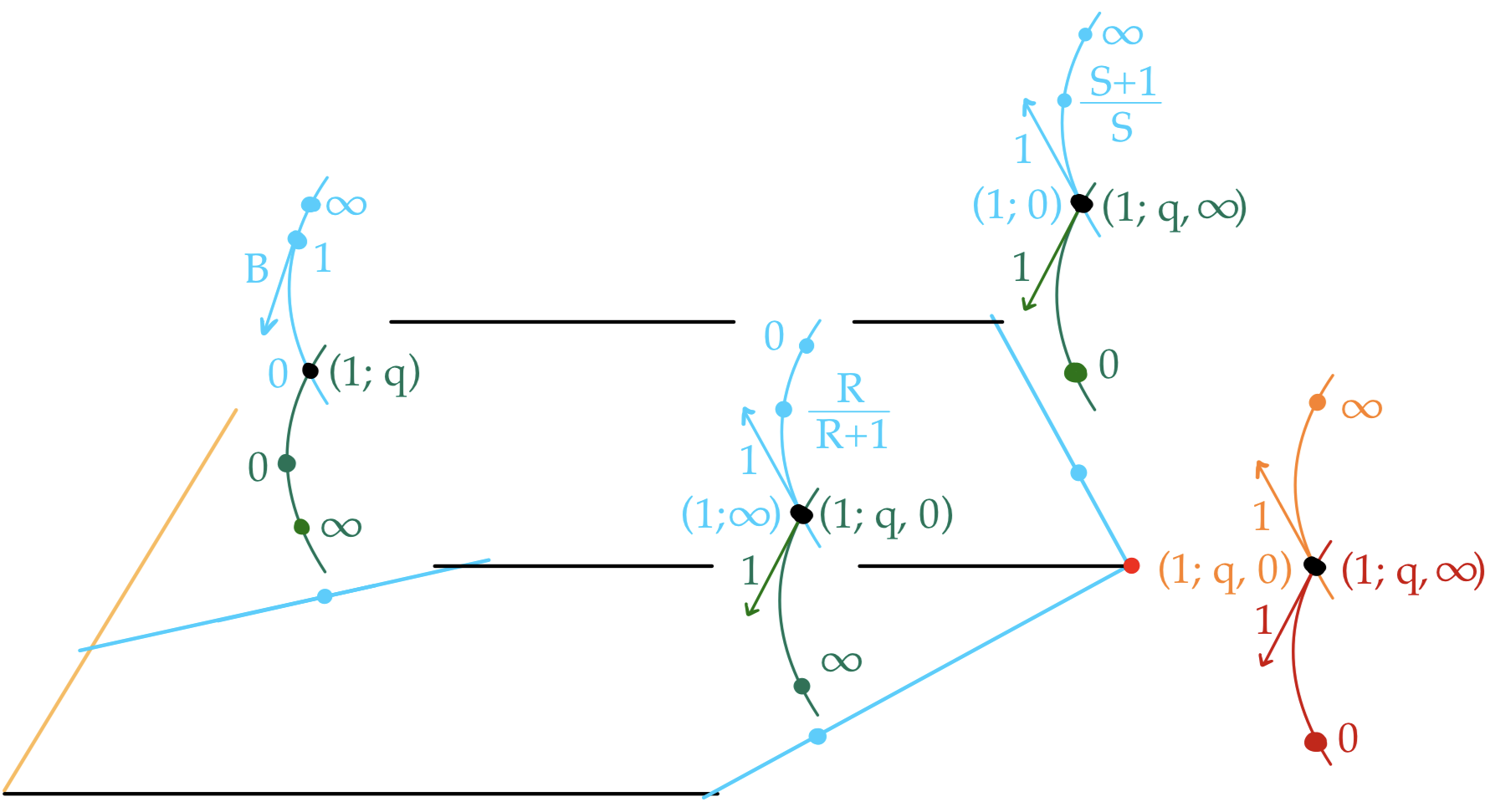}
\end{center}

Let us conclude by resuming in the following picture the coordinates we introduced on $\overline\M$.
\begin{center}
    \includegraphics[width=4cm]{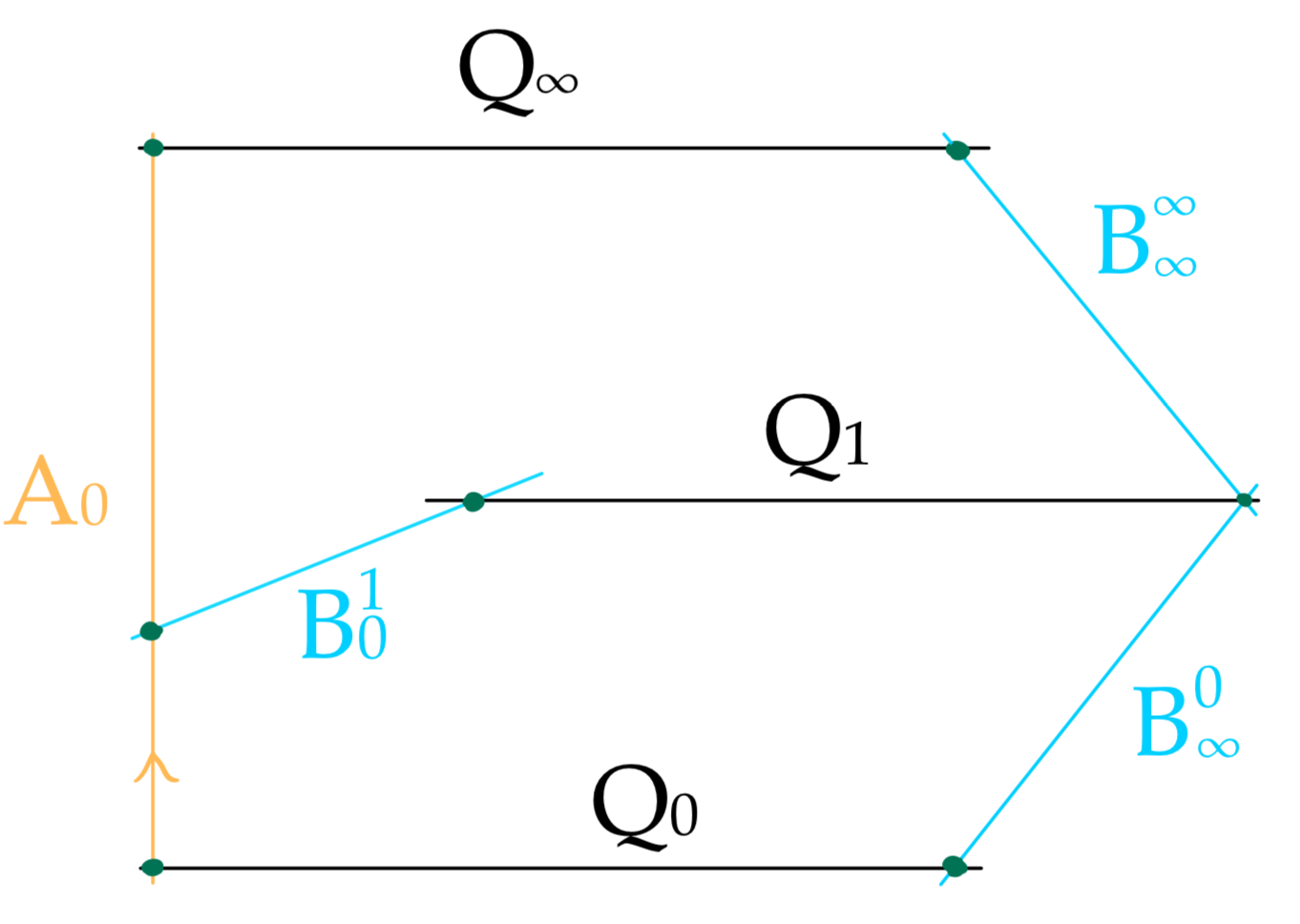}\hskip 40 pt
    \includegraphics[width=4cm]{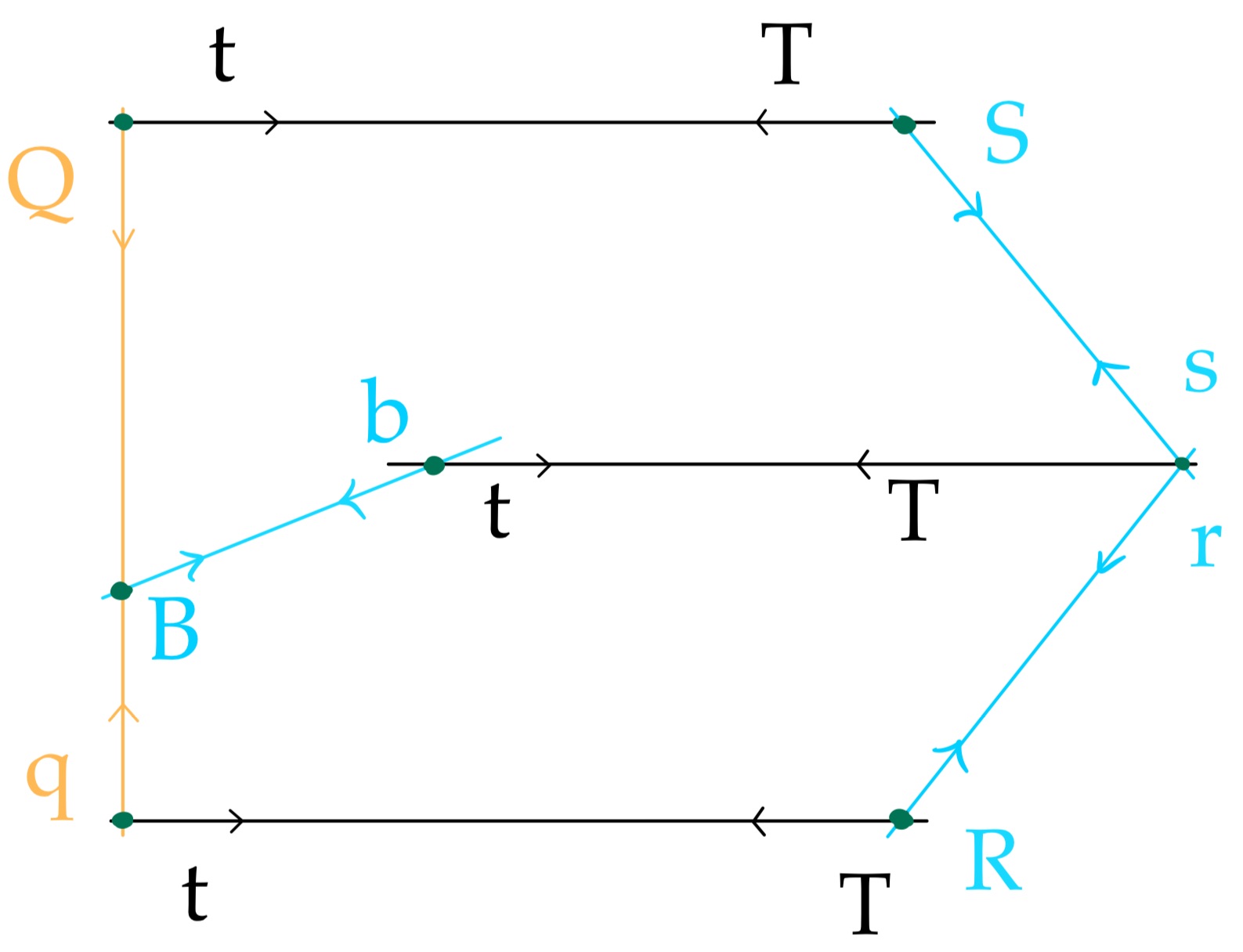}
\end{center}

\subsection{Connections on the Boundary}\label{cononbou}
We already proved in Proposition \ref{openmod} that the moduli space of PV type connection $\Conn$ contains $\M\times\C$ as a dense open set, and we studied in Theorem \ref{compbasis} the compactification of the basis $\M$. In order to construct the compactification $\overline\Conn$, we have to make a choice about the "vertical" coordinate, the coordinate along the fibers, both inside $\Conn$ as along the new boundary components. Note that since we want extend a line bundle, we will mostly work on the smooth surface $\widetilde\M$ instead of the singular $\overline \M$.\\
Notice that a connection on a stable nodal curve is a collection of connections on any smooth component. One of the smooth components is always rigid, in the sense that it is a rational irregular curve endowed with a divisor of degree three (the node becomes a pole of the connection). There we get a (sometimes confluent) hypergeometric connection, that is rigid since there is not any free parameter that we can deform.
\begin{es}
    A stable curve lying on $A_0$ is composed of an lower (in yellow in Figure \ref{A0}) smooth component that is rigid since no free parameters appear on it, and an upper component (green) in which the parameter $q$ appears. A connection on this stable nodal curve is the data of two distinct connections: one on the upper component (Heun connection) and one on the lower (confluent hypergeometric connection). 
\end{es}
On $\M\times \C$ we use the coordinate $\hat p$, that appears in the normal form and that we have already seen in $\F_2$, as in Figure \ref{F2}. On each boundary component we choose a suitable coordinate around a generic point that allows us to compute the resulting confluent connection on the stable nodal curve. For the following, we consider a connection $\nabla$ of PV type, in the normal form as described in section \ref{S1}. We call $\Omega_0$ the connection matrix relative to the trivialising open set $U_0=\P^1\setminus\{\infty\}$. 

\medskip
\textbf{Connections on $A_0$.}
We have a family of stable nodal curves with two smooth components depending on the parameter $q$. We have hence to describe the limit connections on both smooth components. \\
On the first component, relative to setting $t=0$, in green in Figure \ref{A0}, we can easily define the limit connection by directly computing
\[\Omega_0^{t=0}:=\lim_{t\to0}\Omega_0.\]
After applying an elementary transformation based in $x=1$, the double pole disappears. Therefore $\lim_{t\to0}\nabla$ corresponds to an hypergeometric equation with polar divisor $[0]+[1]+[\infty]$. \\
We can study what happens on the other smooth components, that is the one occurring when the Moebius transformation $f_{A_0}$ (\ref{fA0}) is applied. In order to do so, we consider the following limit:
\[\Omega_0^{f_{A_0}}:=\lim_{t\to0}f_{A_0}^*\Omega_0=\begin{pmatrix}
0 & \frac{1}{\left(X -1\right)^{2}} 
\\
\frac{\rho \left(q -1\right)}{X^{2}}-\frac{\hat p \left(\hat p -\kappa_1+1 \right)}{\left(q -1\right) X^{2}}+\frac{\hat p^{2}}{\left(q -1\right)^{2} X^{2}}+\frac{\hat p \left(\hat p +\mathit{\kappa_0} \right)}{q \,X^{2}} & \frac{1}{\left(X -1\right)^{2}}+\frac{\kappa_1}{X}-\frac{\kappa_1}{X -1} 
\end{pmatrix}.
\]
\begin{oss}
    We make some observation that we will not repeat for the others computations that follow.
    \begin{itemize}
        \item The pole of order two in $X=0$ becomes \textcolor{purple}{simple} after applying an elementary transformation in $X=0$: 
        \item The leading coefficient relative to the pole $X=1$ has \textcolor{teal}{now 1 as coefficient} instead of $t$. 
    \end{itemize}
    \[\begin{pmatrix}
0 & \frac{1}{X \left(X -1\right)^{2}} 
\\
 \frac{\rho \left(q -1\right)}{\color{purple}X}-\frac{\hat p \left(\hat p -\kappa_1+1 \right)}{\left(q -1\right) \color{purple}X}+\frac{\hat p^{2}}{\left(q -1\right)^{2} \color{purple}X}+\frac{\hat p \left(\hat p +\mathit{\kappa_0} \right)}{q \,\color{purple}X} & \frac{\color{teal}1}{\left(X -1\right)^{2}}+\frac{\kappa_1-1}{X}-\frac{\kappa_1}{X -1}
\end{pmatrix}.
\]
\end{oss}
Finally, we show that the spectral residual datum at the nodal singularity agrees on both connections, as shown in \cite{strata}.
\begin{prop}
    The two connections have the same residual spectral datum in the node, that is
    \[4 \rho( q-1) +\kappa_1^2-\frac{4 \hat p \left(\hat p -\kappa_1+1 \right)}{q -1}+\frac{4 \hat p^{2}}{\left(q -1\right)^{2}}+\frac{4 \hat p \left(\hat p +\kappa_0 \right)}{q}.\]
\end{prop}
\begin{proof}
   It follows from a direct computation via the formula \ref{formresspec}, in Section \ref{resspecd}.
\end{proof}

\textbf{Connections on $Q^0$.}
Here we have a family of stable nodal curves depending on the parameter $t$. 
In the first component, relative to the direct setting $q=0$, we can try to compute the limit connection
\[\Omega_0^{q=0}:=\lim_{q\to0}\Omega_0(x; t,q,\hat p).\]
\begin{lemma}\label{Q0}
The limit
\[\lim_{q\to0}\Omega_0(x; t,q,\hat p)\]
exists if, and only if, for $q\to0$, the point $(q,\hat p)\in \F_2$ follows the trajectory $(q, \alpha q+\beta)$, for $\alpha\in \C$ and $\beta=0,-\kappa_0$. Moreover, the resulting connection depends on $\alpha$, the asymptotic slope of the convergence.
\end{lemma}
\begin{proof}
The first part of the Lemma follows from Theorem \ref{confl}. To distinguish the different limits, we call $\alpha^0_-$ the slope of the invariant curve for $\beta=0$, and $\alpha^0_+$ the one for $\beta=-\kappa_0$. Computing the limit for $\beta=0$, we get:
    \[\lim_{q\to0}\Omega_0(x; t,q,\alpha^0_- q )=\begin{pmatrix}
0 & \frac{1}{x\left(x -1\right)^{2}} 
\\
 \alpha^0_-  \mathit{\kappa_0} -x \rho  & \frac{-\mathit{\kappa_0} -1}{x}-\frac{\kappa_1}{x -1}+\frac{t}{\left(x -1\right)^{2}} 
\end{pmatrix}
\]
and, for $\beta=-\kappa_0$:
\[\lim_{q\to0}\Omega_0(x; t,q,\alpha^0_+ q -\kappa_0)=\begin{pmatrix}
0 & \frac{1}{x\left(x -1\right)^{2}} 
\\
 2 \mathit{\kappa_0}^{2}+\left(-\alpha^0_+ +t +\kappa_1-\frac{1}{x}\right) \kappa_0 -x \rho & \frac{-\kappa_0 -1}{x}+\frac{-\kappa_1}{x -1}+\frac{t}{\left(x -1\right)^{2}} 
\end{pmatrix}.\]
In both cases the matrix depends, as desired, on the parameter $t$ defining the stable nodal curve and on the parameter $\alpha^0_\pm$ defining the connection.
\end{proof}
\begin{oss}
    We can say that $\alpha^0_\pm$ are the coordinates of the exceptional divisors arising from blowing up the red points in $\F_2$ (represented in green in Figure \ref{Okam}).
\end{oss}
We still have to study what happens on the other smooth component of the stable curve, with coordinate $X=f_{Q^0}(x)$ (\ref{fQ0}). To do so, we consider the matrix $f_{Q^0}^*\Omega_0$ with the parameter $\hat p$.
\begin{lemma}
    The limit $\Omega_0^{f_{Q^0}}:=\lim_{q\to0}f_{Q^0}^*\Omega_0$ produces a well defined connection.
\end{lemma}
\begin{proof}
    Computing the limit we get:
    \[\lim_{q\to0}f_{Q^0}^*\Omega_0(X; t,q,\hat p)=\begin{pmatrix}
0 & -\frac{1}{X \left(X -1\right)}
\\
 -\frac{\hat p}{X-1} +\frac{-\hat p \left(\hat p+\kappa_0  \right)}{(X -1)^2} & -\frac{\kappa_0}{X}+\frac{\kappa_0+1}{X -1}  
\end{pmatrix}.
\]
\end{proof}
A further computation shows that we get two different types of connections. For a generic value of $\hat p$ we get a connection with a ramified pole in $X=1$, while for $\hat p=0, -\kappa_0$ (that are the values appearing in Lemma \ref{Q0}) we get the hypergeometric equations we are expecting.
\begin{prop}
    Connections with $\hat p = 0$ have the same residual spectral data in the node that corresponds to $(\kappa_0+1)^2$. Connections with $\hat p= -\kappa_0$ have also the same residual spectral data in the node that corresponds to $(\kappa_0-1)^2$.
\end{prop}
\begin{proof}
   It follows from a direct computation via the formula \ref{formresspec}, in Section \ref{resspecd}.
\end{proof}

\medskip
\textbf{Connections on $Q^\infty$.}
The same procedure as for $Q^0$ can be applied (see Section \ref{symm}). The vertical coordinate that we should use is $P^\infty=-\frac{\hat p}{(\frac{1}{Q}-1)^2}$, where $Q=1/q$.

\medskip
\textbf{Connections on $B_0^1$.}
We can directly compute the limit for $t\to0$ and $q\to 1$. To do that, we introduced in the last section the coordinate $B=\frac{t}{q-1}$.
\begin{lemma}
    The limit $\lim_{q\to1}\Omega_0(x; B,q,P^1_0)$ produces a well defined connection if we set $P^1_0=-\frac{\hat p}{(q-1)}$.
\end{lemma}
\begin{proof}
    Computing the limit we get: 
    \begin{equation}\label{eqlimb011}
        \lim_{q\to1}\Omega_0(x; B,q, P^1_0)=\begin{pmatrix}
0 & \frac{1}{x \left(x -1\right )^2} 
\\
 (P^1_0)^{2}+P^1_0\left(B -\kappa_1\right) -\rho\left(x -1\right) & -\frac{\kappa_1+1}{x-1}-\frac{\kappa_0}{x} 

\end{pmatrix}
.
\end{equation}
\end{proof}
On the other smooth component we have to apply the Moebius transformation $f_{B_0^1}$.
\begin{lemma}
    The limit $\lim_{q\to1}f_{B^1_0}^*\Omega_0$ produces a well defined connection if we set $P^1_0=-\frac{\hat p}{(q-1)}$.
\end{lemma}
\begin{proof}
    Computing the limit we get: 
    \begin{equation}\label{eqlimb012}
    \lim_{q\to1}f_{B^1_0}^*\Omega_0(x; B,q,P^1_0)=\begin{pmatrix}
0 & \frac{1}{\left(X -1\right)^{2} } 
\\
 \frac{P^1_0}{X}+\frac{P^1_0 \left(B +P^1_0 -\kappa_1\right)}{X^2} & -\frac{\kappa_1}{X -1}+\frac{\kappa_1+1}{X}+\frac{A}{\left(X -1\right)^{2}} 
\end{pmatrix}
.
\end{equation}
\end{proof}
\begin{prop}
    Connections \ref{eqlimb011} and \ref{eqlimb012} have the same residual spectral datum in the node, that is
    \[(\kappa_1+1)^2+4P^1_0(B+P^1_0-\kappa_1).\]
\end{prop}
\begin{proof}
   It follows from a direct computation via the formula \ref{formresspec}, in Section \ref{resspecd}.
\end{proof}
Let us now study connections on those stable curves that present the double pole in the node for the two smooth components. We will see that there will not be equality of the residual spectral data.

\medskip
\textbf{Connections on $Q^1$.}
We expect over $Q^1$ a similar behaviour as over $Q^0$, but the double pole at $x=1$ makes everything more tricky. This time we will start by computing the limit for the pull-back connection $f_{Q^1}^*\Omega_0$. We will introduce a new variable for $Q^1$, and we want to compute the direct limit using that variable.
\begin{lemma}\label{tildeP}
    The limit $\lim_{q\to1}f_{Q^1}^*\Omega_0$ produces a well defined connection if we set $P^1=-\frac{\hat p}{qt}$.
\end{lemma}
\begin{proof}
    Computing the limit we get: 
    \[\lim_{q\to1}f_{Q^1}^*\Omega_0(x; t,q, P^1)=\begin{pmatrix}
0 & \frac{1}{\left(X -1\right)^{2}} 
\\
 \frac{ P^1 \left(( P^1 +2) X -1\right)}{\left(X -1\right)^{2} X} & \frac{1}{\left(X -1\right)^{2}}-\frac{1}{X}+\frac{1}{X -1} 
\end{pmatrix}
.
\]
\end{proof}
We remark that $\rho,\kappa_i$ disappear. It is due to the fact that the limit curve $f_{Q^1}(\P^1)$ for $q\to1$ shows the apparent singularity at infinity and the double pole at one. It means that the relative connection is confluent hypergeometric with trivial monodromy around the poles, since $M_\infty M_1=Id$ and $M_\infty=Id$. \\We now have to rewrite the invariant curves of Theorem \ref{confl} in the new coordinate $ P^1$.
\begin{lemma}\label{invq1P}
    The invariant curves introduced in Theorem \ref{confl} become
    \[P^1=0 \;\;\;\text{ and }\;\;\; P^1=\frac{\kappa_1}{t}(q-1)-1.\]
\end{lemma}
\begin{proof}
    We apply the same procedure of Theorem \ref{confl} after the changing of variable $ P^1=-\frac{\hat p}{qt}$.
\end{proof}
We can now prove the following.
\begin{lemma}\label{Q1}
The limit
\[\lim_{q\to1}\Omega_0(x; t,q, P^1)\]
exists if, and only if, during the limit $q\to1$, the point $(q, P^1)\in \F_2$ follows a trajectory that coincides up to the order two with the curve $(q, \frac{\alpha^1_\pm}{t^2} (q-1)^2+\beta (q-1)+\gamma)$, for $(\beta,\gamma)$ as in Lemma \ref{invq1P}. Moreover the resulting connection depends on $\alpha^1_\pm$.
\end{lemma}
\begin{proof}
    The first part of the Lemma follows from Theorem \ref{confl} and Lemma \ref{invq1P}. Computing the limit for $(\beta,\gamma)=(0,0)$, we get:
    \[\lim_{q\to1}\Omega_0\left(x; t,q,\frac{\alpha^1_-}{t^2} (q-1)^2 \right)=\begin{pmatrix}
0 & \frac{1}{x\left(x -1\right)^{2}} 
\\
 -\rho(x-1)+\alpha^1_-  & \frac{-\kappa_1-1}{x -1}-\frac{\kappa_0}{x}+\frac{t}{\left(x -1\right)^{2}} 
\end{pmatrix}.
\]
and, for $(\beta,\gamma)=\left(\frac{\kappa_1}{t},-1\right)$:
\[\lim_{q\to0}\Omega_0\left(x; t,q,\frac{\alpha^1_+}{t^2} (q-1)^2 +\beta(q-1)+\gamma\right)=\begin{pmatrix}
0 & \frac{1}{x\left(x -1\right)^{2} } 
\\
 \left(\kappa_0 +\frac{1}{x -1}\right) t -\left(x-1 \right) \rho -\alpha^1_+  & \frac{t}{\left(x -1\right)^{2}}+\frac{-\kappa_1-1}{x -1}-\frac{\kappa_0}{x} 
\end{pmatrix}.\]
In both cases the matrices depend, as expected, on the parameter $t$ defining the stable nodal curve and on the parameter $\alpha^1_\pm$ defining the connection, as desired.
\end{proof}
\begin{oss}
    As mentioned before, we were expecting equality of residual spectral data in the node also in this situation, but it does not happens. For the pull-back connection we get 1, while for the direct limits we get respectively $\kappa_1-1$ and $\kappa_1+1$.
\end{oss}

\medskip

\textbf{Connections on $B_\infty^0$.}
Here, both the smooth components of the nodal curve are obtained by a Moebius transformation. We recall that we defined $R=q/T$ the coordinate on $B_\infty^0$, where $T=1/t$.
\begin{lemma}
    The limit $\lim_{T\to0}f_{B_\infty^0}^*\Omega_0$ produces a well defined connection if we set $P_\infty^0=-\frac{\hat p}{(RT-1)^2}$.
\end{lemma}
\begin{proof}
    Computing the limit we get: 
    \[\lim_{T\to0}f_{B_\infty^0}^*\Omega_0(x; T,R, P_\infty^0)=\begin{pmatrix}
0 & -\frac{1}{\left(X -1\right) X} 
\\
 -\frac{P_\infty^0 \left(R +1\right)}{R X -R +X}+\frac{P_\infty^0}{X -1}+\frac{P_\infty^0 \left(P_\infty^0 +R -\kappa_0 \right)}{\left(X -1\right)^{2} R} & \frac{-R -1}{R X -R +X}+\frac{\kappa_0 +1}{X -1}+\frac{1}{\left(X -1\right)^{2}}-\frac{\kappa_0}{X} 
\end{pmatrix}
\]
\end{proof}
And the second one.
\begin{lemma}
    The limit $\lim_{T\to0}g_{B_\infty^0}^*\Omega_0$ produces a well defined connection if we set $P_\infty^0=-\frac{\hat p}{(RT-1)^2}$.
\end{lemma}
\begin{proof}
    Computing the limit we get: 
    \[\lim_{T\to0}g_{B_\infty^0}^*\Omega_0(x; T,R, P_\infty^0)=\begin{pmatrix}
0 & \frac{1}{\left(X -1\right)^{3}} 
\\
 -\rho(X-1) +P_\infty^0 +\frac{P_\infty^0 \left(P_\infty^0 -\kappa_0 \right)}{R} & \frac{1}{\left(X -1\right)^{2}}+\frac{-\kappa_0 -\kappa_1-1}{X -1} 
\end{pmatrix}
\]
\end{proof}
\begin{oss}
    As mentioned before, we were expecting equality of residual spectral data in the node also in this situation, but it does not happens. We get respectively
    \[2 P^0_\infty -\kappa_0 -1+\frac{2 P^0_\infty \left(P^0_\infty -\kappa_0 \right)}{R}\;\;\;\text{ and }\;\;\;-2 P^0_\infty +\kappa_0 +\kappa_1+1-\frac{2 P^0_\infty \left(P^0_\infty -\kappa_0 \right)}{R}\]
\end{oss}

\medskip

\textbf{Connections on $B_\infty^\infty$.}
The same procedure as for $B^0_\infty$ can be applied (see Section \ref{symm}). The vertical coordinate that we should use is $P^\infty_\infty=-\frac{\hat p}{(\frac{1}{ST}-1)^2}$, where $S=\frac{Q}{T}$ is the coordinate on $B_\infty^\infty$, $Q=\frac1q$, $T=\frac1t$.

\medskip

\textbf{Connections on $A_\infty$.}
The limit for $t\to\infty$ is more tricky than the other limits we studied before. We have indeed two possible approaches to compute it. In Section \ref{seccomp} we will give a geometrical interpretation to both of them.\\The first approach is to consider the two Moebius transformations we used to derive the stable nodal curve. 
\begin{lemma}
    The limit $\lim_{T\to0}f_{A_\infty}^*\Omega_0$ produces a well defined connection if we set 
    $P_\infty^+=-\frac{\hat p}{(q-1)^2}$ or 
    $P_\infty^-=-\frac{\hat p}{(q-1)^2}+\frac{T\kappa_1(q-1)-q}{T(q-1)^2}$.
\end{lemma}
\begin{proof}
    Computing the limit we get: 
    \[\lim_{T\to0}f_{A_\infty}^*\Omega_0(x; T,q, P_\infty^+)=\begin{pmatrix}
0 & -\frac{1}{X\left(X -1\right) } 
\\
 \frac{P_\infty^+}{\left(X -1\right)^{2}} & -\frac{\kappa_0}{X}+\frac{\kappa_0}{X -1}+\frac{1}{\left(X -1\right)^{2}} 
\end{pmatrix}
\]
\[
\lim_{T\to0}f_{A_\infty}^*\Omega_0(x; T,q, P_\infty^-)=\begin{pmatrix}
0 & -\frac{1}{X\left(X -1\right) } 
\\
 \frac{-P^-_\infty+\kappa_0+\kappa_1-1}{(X-1)^2}
 & -\frac{\kappa_0}{X}+\frac{\kappa_0}{X -1}+\frac{1}{\left(X -1\right)^{2}} 
\end{pmatrix}
\]
\end{proof}
\begin{lemma}\label{Ppm}
    The limit $\lim_{T\to0}g_{A_\infty}^*\Omega_0$ produces a well defined connection if we set $P_\infty^+=-\frac{\hat p}{(q-1)^2}$ or $P_\infty^-=-\frac{\hat pT-q}{T(q-1)^2}$.
\end{lemma}
\begin{proof}
    Computing the limit and applying an elementary transformation we get to, respectively
    \[\lim_{T\to0}g_{A_\infty}^*\Omega_0(x; T,q, P_\infty^+)=\begin{pmatrix}
0 & \frac{1}{\left(X -1\right)^{3}} 
\\
 -\rho(X-1)+P_\infty^+   & \frac{-\kappa_0 -\kappa_1-1}{X -1}+\frac{1}{\left(X -1\right)^{2}} 
\end{pmatrix}
\]
\[
\lim_{T\to0}g_{A_\infty}^*\Omega_0(x; T,q, P_\infty^-)=\begin{pmatrix}
0 & \frac{1}{\left(X -1\right)^{3}} 
\\
 -\rho\left(X-1\right)  -P_\infty^- +\kappa_0 +\kappa_1 
 & \frac{-\kappa_0 -\kappa-1}{X -1}+\frac{1}{\left(X -1\right)^{2}} 
\end{pmatrix}
\]
\end{proof}
\begin{oss}
    As mentioned before, we were expecting equality of residual spectral data in the node also in this situation, but it does not happens. We get respectively for the "+" curve
    \[2 P_\infty^+ -\kappa_0 \;\;\;\text{ and }\;\;\;-2 P_\infty^+ +\kappa_0 +\kappa_1+1,\]
    while for the "-" curve we get
    \[2 P_\infty^- -\kappa_0 -\kappa_1+1-2\frac{\kappa_1}{q-1}\;\;\;\text{ and }\;\;\;-2 P_\infty^- +\kappa_0 +2\left(\kappa_1-1\frac{\kappa_1}{q-1}\right).\]
\end{oss}
We can also have a different approach, in which we use the coordinate $P^1$. It is the coordinate used for computing limits in $Q^1$. Let us consider a Moebius transformation that fixes 0 and sets $t=1$. We have a family depending on one parameter of such transformation and we can for instance choose to fix $-1$ (it suffices not to fix 0, $\infty$ or $q$). The curve we get after the limit $t=\infty$ is no more stable, because only the double pole in 1 survived. Anyway, the limit we get is still interesting.
\begin{prop}\label{Ainfp}
    Let us consider the Moebius transformation 
    \[h_{A_\infty}(x):=\frac{(t+1)x+t-1}{(t-1)x+t+1}.\]
    The limit $\lim_{T\to0}h^*_{A_\infty}\Omega_0(X; T,q,P^1)$ produces a well defined connection if we use $ P^1=-\frac{\hat p T}{q}$, the same coordinate as in $Q^1$.
\end{prop}
\begin{proof}
    Computing the limit we get to:
    \[
    \Omega_0^{h_{A_\infty}}:=\begin{pmatrix}
0 & -\frac{1}{\left(X -1\right)^{2}} 
\\
 \frac{4 q  P^1 \left( P^1 +1\right)}{\left(q -1\right)^{2} \left(X -1\right)^{2}} & \frac{1}{\left(X -1\right)^{2}} 
    \end{pmatrix}
    \]
\end{proof}
\begin{oss}
    The relation between the coordinates we used is the following
    \[P_\infty^+=-\frac{\hat p}{(q-1)^2}=\frac{ P^1q}{T(q-1)^2}\;\;\;\text{ and }\;\;\;P_\infty^-=-\frac{\hat p}{(q-1)^2}+\frac{T\kappa_1(q-1)-q}{T(q-1)^2}=\frac{q(P^1+1)}{T(q-1)^2}-\frac{\kappa_1}{q-1}.\]
    We will give a geometric interpretation of these coordinates in the end of the next section.
\end{oss}

\subsection{Extension of the line bundle}

We recall that our final goal is to compactify the space 
$\Conn$. We proceed in several steps. First, in this section, we study how the line bundle structure extends to the boundary components of $\widetilde M$, defining a line bundle $\O_{\widetilde M}(D)$ associated to a divisor $D\in \mathrm{Div}(\widetilde \M)$. We then compactify this space and analyse its birational geometry. This analysis is necessary because the space we ultimately want to work with is the singular variety where the surface $\mathcal A_\infty:=\overline{\O_{ A_\infty}(D)}$ has been contracted producing only canonical singularities.

The strategy is to deduce the divisor $D$ by studying the poles and zeros of the extension of the global section $\hat p =1$ defined in $\Conn$. In the previous section we have studied the vertical coordinate in the generic point of any boundary component. To compute the divisor, it is more suitable to define vertical coordinates based on each intersection point of the boundary components. Some computations show that one possible choice is the following:
\begin{center}
    \includegraphics[width=12cm]{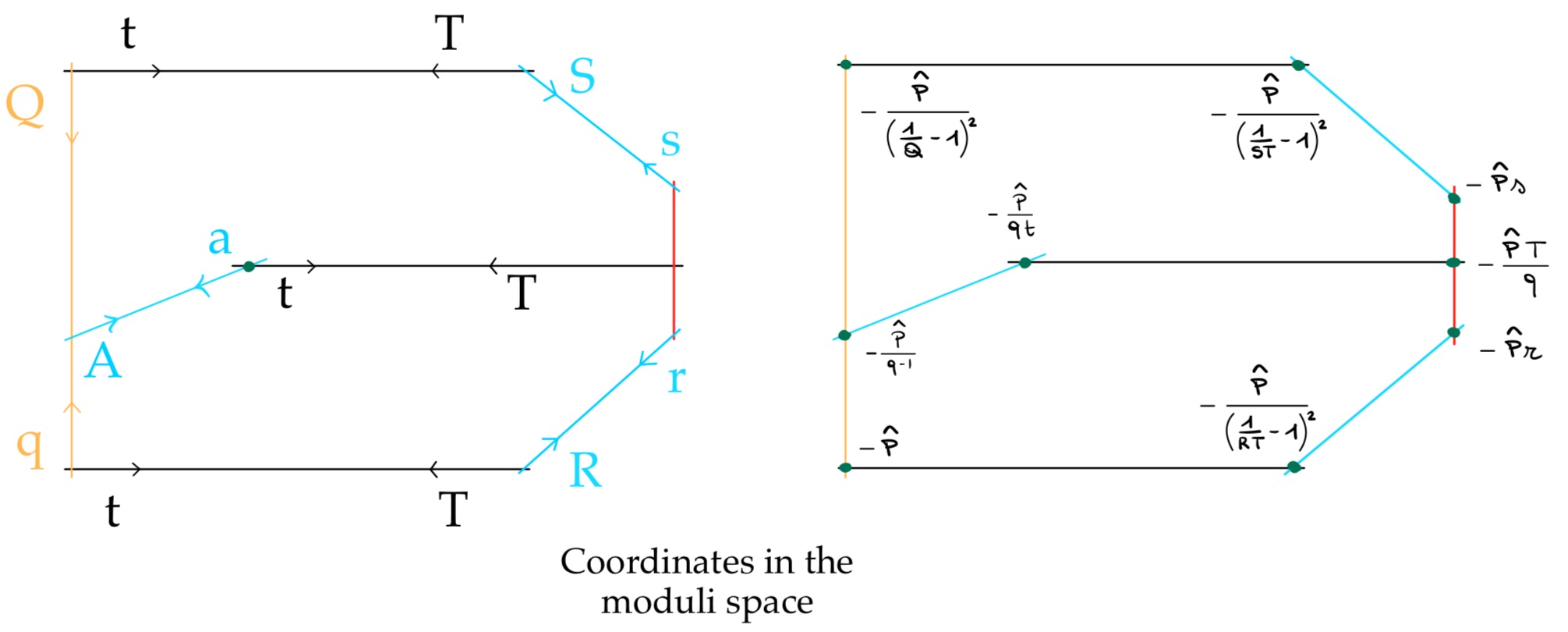}
\end{center}
Before going into the proof, let us recall that in Proposition \ref{Msing} we have shown that $\overline\M$ is a singular space, due to the fact that the $(-2)$-curve $A_\infty$ has been contracted. We will hence in a first moment extend the line bundle on the smooth surface $\widetilde{\M}$, and then we show that it is possible to contract the surface lying over $A_\infty$. In what follows we denote by $\mathcal Q^i$ and $\mathcal A_j$ the total spaces of the restriction of the line bundle over, respectively, $Q^i$ and $A_j$.
\begin{prop}
    The trivial line bundle $\M\times \C$ extends to $\widetilde \M$ as $\O_{\widetilde{\mathcal M}}(D)$, where $D\in \mathrm{Div}(\widetilde \M)$ is the divisor
    \[D=\mathrm{div}(\hat p)\equiv A_0+2Q^1+2B_0^1-B_\infty^0-B_\infty^\infty.\]
\end{prop}
\begin{proof}
    Since it is easier for us to work with a smooth variety, we study in a first moment the extension of the trivial vector bundle $\mathcal M\times \C$ on $\widetilde{\M}$.\\
    By studying the poles and zeros of the section $\hat p=1$ in all the coordinates, we get to the divisor
    \[- B^1_0+2 Q^\infty+2 B^\infty_\infty+ A_\infty. \]
    To get the desired expression, we can explicitly study the Picard group of $\widetilde{\M}$. \\We now it is the blowup of $\P^1\times\P^1$ in three points. Therefore, since $\mathrm{Pic}(\P^1\times\P^1)\cong\Z\times\Z$, we have that $\mathrm{Pic}(\widetilde{\M})\cong \Z^5$ and is is for instance generated by
    \[\mathrm{Pic}(\widetilde{\M})\cong [A_0]\cdot\Z+[Q^0]\cdot\Z+[B^1_0]\cdot\Z+[B^0_\infty]\cdot\Z+[B^\infty_\infty]\cdot\Z.\]
    Moreover we know that in $\P^1\times\P^1$ the divisors $A_0$ and $A_\infty$ are equivalent, and the same holds for $Q^0, Q^1$ and $Q^\infty$. Since we know that a holomorphic function $f\colon\P^1\times\P^1\to\C$ vanishing at a point $p\in\P^1\times\P^1$, vanishes also on the exceptional divisor $E_p\subseteq\mathrm{Bl}_p(\P^1\times\P^1)$, if we remark that $A_i=\{t-i=0\}$ and $Q^j=\{q-j=0\}$, we get the following equivalences
    \[\begin{matrix}
        \begin{cases}
        A_0+B^1_0=A_\infty+B_\infty^0+B^\infty_\infty\\
        Q^0+B_\infty^0=Q^1+B^1_0=Q^\infty+B^\infty_\infty
    \end{cases}
    &\implies&
    \begin{cases}
        A_\infty=A_0+B_0^0-B_\infty^0-B^\infty_\infty\\
        Q^0=Q^1+B^1_0-B_\infty^0\\
        Q^\infty=Q^1+B^1_0-B^\infty_\infty
    \end{cases}
    \end{matrix}\]
    And if we replace them in the expression of $D$ we get
    \begin{align*}
        D&=- B^1_0+2 (Q^1+B^1_0-B^\infty_\infty)+2 B^\infty_\infty+ (A_0+B_0^1-B_\infty^0-B^\infty_\infty)=\\&=A_0+2B^1_0+2Q^1-B^\infty_\infty-B^0_\infty
    \end{align*}
    as desired.
\end{proof}
Since the base space we would like to work with is $\overline \M$, that is the contraction of $A_\infty$ inside $\widetilde \M$, we wonder if we can contract the surface $\mathcal  A_\infty:=\overline{\O_{ \widetilde \M}(D)_{|A_\infty}}$. \\ Before doing that, we consider the compactification of the line bundle, in order to work with a compact algebraic 3-manifold. 
\begin{prop}
    The line bundle $\O_{ \widetilde \M}(D)$ is trivial over $Q^0$, $Q^1$, $A_\infty$ and $Q^\infty$.
\end{prop}    
\begin{proof}
    We recall that for any pair of divisor $D$ and $D'$ it holds that
    \[\deg\O_{D'}(D)=(D\cdot D').\]
    In particular:
    \begin{align*}
         &(D\cdot Q^0)=\Big((- B^1_0+2 Q^\infty+2 B^\infty_\infty+ A_\infty)\cdot Q^0\Big)=0+0+0+0=0,\\
         &(D\cdot Q^1)=\Big((- B^1_0+2 Q^\infty+2 B^\infty_\infty+ A_\infty)\cdot Q^1\Big)=-2+0+0+2=0,\\
        &(D\cdot Q^\infty)=\Big((- B^1_0+2 Q^\infty+2 B^\infty_\infty+ A_\infty)\cdot Q^\infty\Big)=0-2+2+0=0,\\
         &(D\cdot A_\infty)=\Big((- B^1_0+2 Q^\infty+2 B^\infty_\infty+ A_\infty)\cdot A_\infty\Big)=0+0+2-2=0,
    \end{align*}
    as desired.
\end{proof} 
\begin{oss}
    We already expected this result the for $Q^i$'s: indeed, Lemmas \ref{Q0} and \ref{Q1} give us the expression of some non vanishing global sections. In particular, we also know that $J^1\O_{ \widetilde \M}(D)$ is trivial over $Q^1$.
\end{oss}
We can hence choose the following vertical coordinates that trivialise the bundle on those boundary components:
\[P^0=-\frac{\hat p}{(q-1)^2}=-\frac{\hat p}{(RT-1)^2},\;\;\;\;\; P^1=-\frac{\hat p}{qt}=-\frac{\hat pT}{q}=-\hat p s=-\hat pr,\;\;\;\;\; P^\infty=-\frac{\hat p}{(\frac{1}{Q}-1)^2}=-\frac{\hat p}{(\frac{1}{ST}-1)^2}.\]
It suffices then to choose a coordinate at the intersection $A_0\cap B^1_0$, that is $P^1_0=-\frac{\hat p}{q-1}$.
\begin{center}
    \includegraphics[width=5cm]{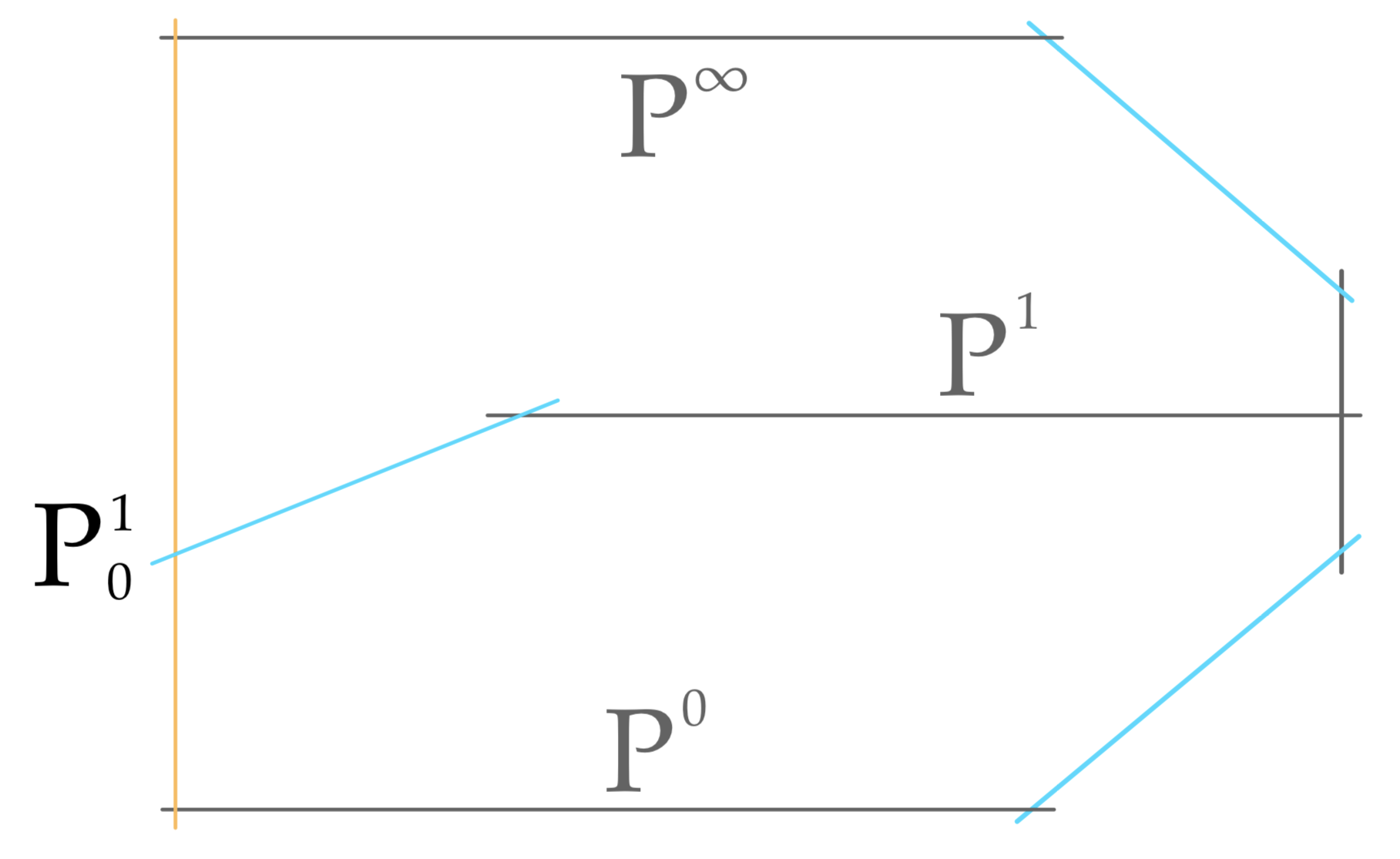}
\end{center}
With these coordinates, the sections in $\mathcal Q^0$ are given by the equations $P^{0}=0,P^{0}=-\kappa_0$ and in $\mathcal Q^1$ are given by the equations $P^{1}=0,P^{1}=-1$. 

The following result shows that we can contract $\mathcal{A}_\infty$ producing mild singularities. It has the advantage to describe the line bundle over $\overline\M$, getting rid of $\widetilde \M$. However, we have already seen that we need to blow up some sections in $\mathcal Q^1$ and $\mathcal A_\infty$ to build the moduli space $\overline \Conn$, making the contraction of $\mathcal A_\infty$ way more complicated.
\begin{prop}
    Let us denote $X:=\overline{\O_{ \widetilde \M}(D)}$. The contraction of $\mathcal A_\infty$ inside $X$ produces canonical singularities.
\end{prop}
\begin{proof}
    First of all let us notice that, since $\O_{ A_\infty}(D)\cong\O$, then $\mathcal A_\infty:=\overline{\O_{ A_\infty}(D)}\cong\P(\O\oplus\O)\cong \P^1\times\P^1$. \\
    We want then to contract the "horizontal" fibers of $\mathcal A_\infty$, and then we have to show that they are $K_X$-trivial, where $K_X$ is the canonical bundle of $X$. Let us call $C_a=\P^1\times\{a\},C_b=\P^1\times\{b\}$ two of those fibers. They are linearly equivalent in $X$ since they are in $\mathcal A_\infty$ and $\dim \mathcal A_\infty=2$. We identify $A_\infty$ with $C_0$. It is then sufficient to show that $(K_X\cdot A_\infty)=0$. Since $A_\infty\subseteq\widetilde M$, it holds that $(K_X\cdot A_\infty)=(K_X|_{\widetilde \M}\cdot A_\infty)$.\\By the adjunction formula we have that 
    \[K_X|_{\widetilde \M}=K_{\widetilde \M}-\widetilde M,\]
    and therefore
    \[(K_X\cdot A_\infty)=(K_{\widetilde \M}\cdot A_\infty)-(\widetilde M\cdot A_\infty).\]
    We can compute the first product since we know the canonical bundle of $\widetilde  \M$: 
    \[(K_{\widetilde \M}\cdot A_\infty)=((2(Q^0+A_0)-B^1_0-B^\infty_0+B^\infty_\infty)\cdot A_\infty)=0\]
    and the second product is zero since $A_\infty\subseteq \widetilde \M$.\\We have then shown that all the horizontal fibers of $\mathcal A_\infty$ are $K_X$-trivial. We have now to compute their discrepancies in order to conclude the proof. Let $\hat \pi\colon X\to \hat X$ the contraction map. Then,
    \[K_X=\hat\pi^*K_{\hat X}+\alpha\mathcal A_\infty\]
    and the equality must hold also when intersecting with $A_\infty$:
    \[0=(K_X\cdot A_\infty)=(\hat\pi^*K_{\hat X}\cdot A_\infty)+\alpha(\mathcal A_\infty\cdot A_\infty).\]
    The first intersection product is zero, since $A_\infty$ has been contracted, while we can compute the second via the formula $(\mathcal A_\infty\cdot A_\infty)=-2-(A_\infty)^2_{\mathcal A_\infty}=-2+0$, since $\mathcal A_\infty\cong \P^1\times\P^1$ and $A_\infty$ is just a fiber. We conclude hence that $\alpha=0$ and then that the contraction process produces only canonical singularities.
\end{proof}

\subsection{The compactified moduli space}\label{seccomp}
We immediately remark that $\overline{\O_{ \widetilde \M}(D)}$ is not the moduli space we are looking for
. Indeed, there are points inside it that do not represent any PV type connection (as points on the $\mathcal Q^i$'s outside the two special sections, or most of the points in $\mathcal A_\infty$) and there are some points that represent infinite connections (as points in the special sections in the $\mathcal Q^i$'s).\\
As shown in Lemma \ref{Q0}, for $\mathcal Q^0$ (and the same holds for $\mathcal Q^\infty$) the connections appear only on the special sections $P_{Q^0}=0,-\kappa_0$. We recall that the resulting limits depend on the slope of incidence, and therefore each point of these sections encodes a one parameter family of connections. We can then blow-up these sections and ignore the other part of $\mathcal Q^0$. For $\mathcal Q^1$ the process is similar, but, following the Lemma \ref{Q1}, we need to blow up two times each section to make the connections appear. We can then ignore $\mathcal Q^1$ and the first exceptional divisor of each section. To have a picture in mind, we have exactly a family of Okamoto spaces as in Figure \ref{Okam} for each choice of $t\in \C^*$. In $t=0$ we have the hypersurfaces $\mathcal A_0$ and $\mathcal B^1_0$, while in $t=\infty$ the situation is more complicated.
\begin{oss}
    We get a new sort of Okamoto divisor given by $2[\mathcal S_\infty]+[\mathcal Q^0]+[\mathcal Q^1]+[\mathcal Q^\infty]+[\mathcal E^+]+[\mathcal E^-]$ that we can visualise by considering in families the surfaces of Figure \ref{Okam}, with the differences given by the use of a different vertical coordinate over $\mathcal Q^1$.
\end{oss}
\[P_\infty^+=\frac{ P^1q}{T(q-1)^2}\;\;\;\text{ and }\;\;\;P_\infty^-=\frac{q(P^1+1)}{T(q-1)^2}-\frac{\kappa_1}{q-1}.\]
it corresponds geometrically to blow up the sections $ P^1=0,-1$ in $\mathcal A_\infty$.
\begin{center}
    \includegraphics[width=7cm]{CPXGEOM.jpeg}
\end{center}
We can resume this section in the following result. 

\begin{thm}\label{thm:fibt}
	The compactified moduli space $\overline\Conn$ comes with a fibration $\pi\colon\overline\Conn\xrightarrow{t} \P^1$ such that
	\begin{itemize}
		\item[i)] For any $t\in \C^*$, the fiber $\pi^{-1}(t)$ is an Okamoto space as in Figure \ref{Okam}, that is a 8-blow-up of the second Hirzebruch space $\F_2$.
		\item[ii)] The surface $\pi^{-1}(0)$ is given by $\mathcal A_0\cup\mathcal B^1_0$. The component $\mathcal A_0$ is isomorphic to $\F_1$, while $\mathcal B^1_0$ is isomorphic to a double blow up of $\F_1$.
		\item[iii)] Let us denote by $\mathcal F_\infty^+$ and $\mathcal F_\infty^-$ the blow up of the two sections $P^1=0$ and $P^1=-1$ in $\mathcal A_\infty$. The surface $\pi^{-1}(\infty)$ is given by $\mathcal B^0_\infty\cup\mathcal B^\infty_\infty\cup\mathcal F_\infty^+\cup\mathcal F_\infty^-$. The surfaces $\mathcal B^0_\infty$ and $\mathcal B^\infty_\infty$ are isomorphic to a double blow up of $\F_1$.
	\end{itemize}
\end{thm}
\begin{proof}
	Assertion \textit{(i)} has been proved in the previous sections. We need to show that surfaces $\mathcal B^i_j$s are isomorphic to a double blow up of $\F_1$, and that $\mathcal A_0\cong \F_1$. It is equivalent to show that the line bundles $\O_{\widetilde \M}(D)_{| B^i_j}$ and $\O_{\widetilde \M}(D)_{| A_0}$ have degree equal to one. 
		\begin{align*}
		&(D\cdot B^1_0)=\Big((A_0+2Q^1+2B_0^1-B_\infty^0-B_\infty^\infty)\cdot B^1_0\Big)=1+2-2+0+0=1,\\
		&(D\cdot B^0_\infty)=\Big((A_0+2Q^1+2B_0^1-B_\infty^0-B_\infty^\infty)\cdot B^0_\infty\Big)=0+0+0+1+0=1,\\
		&(D\cdot B^\infty_\infty)=\Big((A_0+2Q^1+2B_0^1-B_\infty^0-B_\infty^\infty)\cdot B^\infty_\infty\Big)=0+0+0+0+1=1,\\
		&(D\cdot A_0)=\Big((A_0+2Q^1+2B_0^1-B_\infty^0-B_\infty^\infty)\cdot A_0\Big)=-1+0+2+0+0=1,
	\end{align*}
	as desired.
\end{proof}

\subsection{Symmetries}\label{symm}
The construction we made presents some symmetries. The most visible one is the involution $\Phi_{0,\infty}$ exchanging the role of 0 and $\infty$. There are tree others involutions $\Psi_a$, for $a=0,1,\infty,$ exchanging the residual spectral data $\kappa_a^+\leftrightarrow\kappa_a^-$ of each singularity. We can express them in coordinates, on the Zariski open set $\M\times\C\subseteq \Conn$. They are well defined on the whole compactification of the moduli space.
\[\Phi_{0,\infty}=\begin{cases}\kappa_\infty=\kappa_0\\\kappa_0=\kappa_\infty\\q=\frac1q\\\kappa_1=-\kappa_1\\\hat p=\frac q2\frac{-2\hat pq^3+\kappa_0q^2+\kappa_\infty q^2+4\hat pq^2+(\kappa_1-1)q^2-2\kappa_0q-2\kappa_\infty q-2\hat pq-2qt+\kappa_0+\kappa_\infty+2q-\kappa_1-1}{(q-1)^2}\end{cases}\]
and 
\[\Psi_{0}=\begin{cases}t=t\\\kappa_0=-\kappa_0\\\hat p=\hat p-\frac{\kappa_0}{q}\end{cases}\;\;\;\;\Psi_{1}=\begin{cases}t=-t\\\kappa_1=-\kappa_1\\\hat p=\hat p+\frac {t}{(q-1)^2}-\frac{\kappa_1}{q-1}\end{cases}\;\;\;\;\Psi_{\infty}=\begin{cases}t=t\\\kappa_\infty=-\kappa_\infty\\\hat p=\hat p\end{cases}\]
\begin{prop}
    The involutions $\Phi_{0,\infty},\Psi_{0},\Psi_{1}$ and $\Psi_{\infty}$ generate a transformation group of order 16 acting on $\overline\Conn$. Moreover,
    \[\Phi_{0,\infty}\circ\Psi_0=\Psi_\infty\circ \Phi_{0,\infty}.\]
\end{prop}
\begin{proof}
Explicit computations show that $(\Phi_{0,\infty})^2=(\Psi_0)^2=(\Psi_1)^2=(\Psi_\infty)^2=Id$, and that $\Phi_{0,\infty}\circ\Psi_0=\Psi_\infty\circ \Phi_{0,\infty}$, for all $a=0,1,\infty$.
\end{proof}
Geometrically, the symmetries $\Psi_i$ correspond to exchange the two exceptional divisors in $\mathcal Q^i$. For instance, in a neighbour of $\mathcal Q^0$, the two invariant curves $\hat p=0$ and $\hat p = \alpha q-\kappa_0$ are switched, and the same happens for the others $\mathcal Q^i$. The symmetry $\Phi_{0,\infty}$ exchanges $\mathcal Q^0$ and $\mathcal Q^\infty$. This is the reason why, in the previous sections, we often skipped the computations for $\mathcal Q^\infty$.

\newpage
\bibliographystyle{plain}
\bibliography{bibliography}

\end{document}